\documentclass[11pt, reqno]{amsart}
\usepackage{amsmath,amsthm,amsfonts,amssymb,mathrsfs,bm,graphicx,stmaryrd}
\usepackage{subcaption}\usepackage{mathtools}
\usepackage{dsfont}
\usepackage{enumitem}
\usepackage{multicol}
\usepackage[colorlinks=true,linkcolor=blue,citecolor=blue]{hyperref}
\usepackage{xcolor}
\usepackage{bbm}

\newcommand{\wcL}{\widetilde{\mathcal{L}}}

\newcommand{\wGamma}{\widetilde{\Gamma}}
\newcommand{\fh}{\mathfrak{h}}
\newcommand{\cI}{\mathcal{I}}
\newcommand{\RR}{\mathbb{R}}

\newcommand{\PP}{\mathbb{P}}
\newcommand{\EE}{\mathbb{E}}

\newcommand{\wcB}{\widetilde{\cB}}

\newcommand{\0}{\mathbf{0}}

\newcommand{\ualpha}{\underline{\alpha}}

\newcommand{\cJ}{\mathcal{J}}

\newcommand{\tL}{\mathtt{L}}
\newcommand{\tR}{\mathtt{R}}

\newcommand{\NN}{\mathbb{N}}

\usepackage[letterpaper,hmargin=1in,vmargin=1in]{geometry}
\parskip 	\smallskipamount

\newtheorem{theorem}{Theorem}
\newtheorem{question}[theorem]{Question}
\newtheorem{lemma}[theorem]{Lemma}

\newtheorem{definition}[theorem]{Definition}

\newtheorem{proposition}[theorem]{Proposition}

\theoremstyle{definition}

\newcommand{\cR}{\mathcal{R}}

\newcommand{\cA}{\mathcal{A}}
\newcommand{\cB}{\mathcal{B}}

\newcommand{\cC}{\mathcal{C}}
\newcommand{\cD}{\mathcal{D}}
\newcommand{\cE}{\mathcal{E}}

\newcommand{\cW}{\mathcal{W}}
\newcommand{\cL}{\mathcal{L}}

\newcommand{\cS}{\mathcal{S}}

\newcommand{\inte}{\mathrm{int}}
\newcommand{\g}{\mathrm{g}}

\newcommand{\ZZ}{\mathbb{Z}}
\newcommand{\LL}{\mathbb{L}}

\newcommand{\cH}{\mathcal{H}}
\newcommand{\oalpha}{\overline{\alpha}}
\newcommand{\cT}{\mathcal{T}}
\usepackage[backend=biber,doi=false,style=alphabetic,url=false,maxalphanames=10,maxnames=50]{biblatex}
\AtBeginBibliography{\small}

\begin{document}

\title[]{The metric removability of interfaces in the directed landscape}
\author[]{Manan Bhatia}
\address{Manan Bhatia, Department of Mathematics, Massachusetts Institute of Technology, Cambridge, MA, USA}
\email{mananb@mit.edu}
\date{}
\maketitle
\begin{abstract}
  The directed landscape is a prominent model of random geometry which is believed to be the universal scaling limit of all planar random geometries in the Kardar-Parisi-Zhang universality class. It comes equipped with a few different natural simple curves associated to it, such as geodesics and interfaces. Given such a curve, one might wonder whether the geometry off this curve determines the entire landscape, or if in fact, there is non-trivial extra information actually present ``on'' the curve. In this paper, we show that the former is true for an interface in the directed landscape, while the latter is true for a geodesic instead. Further, as is used in the proof of the first assertion above, we show that the set of times where any geodesic intersects an interface a.s.\ has dimension zero.
\end{abstract}
\tableofcontents
\section{Introduction}
\label{sec:intro}
The directed landscape, constructed in the work \cite{DOV18} as the scaling limit of a certain integrable last passage percolation model, is a prominent continuous model of random geometry. The importance of the directed landscape stems from its expected universality-- the belief that even though it was constructed via integrability, and is known \cite{DV21} to be the scaling limit of a few integrable last passage percolation models, it is in fact the universal space-time scaling limit of all the discrete random geometries and growth models in the (1+1)-dimensional Kardar-Parisi-Zhang (KPZ) \cite{KPZ86} universality class. 

Formally, the directed landscape is a random continuous function from the space
\begin{equation}
  \label{eq:95}
  \RR_{\uparrow}^4=\{(x,s;y,t)\in \RR^4: s<t\}
\end{equation}
to $\RR$ and satisfies, for any $s<r<t$, the composition law
\begin{equation}
  \label{eq:96}
  \cL(x,s;y,t)=\max_{z\in \RR}\{ \cL(x,s;z,r)+ \cL(z,r;y,t)\}.
\end{equation}
Intuitively, $\cL(x,s;y,t)$ should be interpreted as the directed ``distance'' from $(x,s)$ to $(y,t)$, in a geometry where paths are only allowed to traverse upwards in time, where we think of the time axis as the vertical axis and the space axis as the horizontal axis. Indeed, for any continuous function (thereafter simply called a path) $\psi\colon [s,t]\rightarrow \RR$ with $\psi(s)=x$ and $\psi(t)=y$, we can define the length
\begin{equation}
  \label{eq:97}
  \ell(\psi;\cL)=\inf_{k\in \NN} \inf_{s=r_0<r_1<\cdots<r_k=t}\sum_{i=1}^k\cL(\psi(r_{i-1}),r_{i-1};\psi(r_{i}),r_{i}),
\end{equation}
and then almost surely, for every $(x,s;y,t)\in \RR_\uparrow^4$, we have the equality $\cL(x,s;y,t)=\sup_\psi \ell(\psi;\cL)$, where the supremum is over all paths $\psi$ from $(x,s)$ to $(y,t)$ as above. In fact, almost surely, for all $(x,s;y,t)\in \RR_{\uparrow}^4$, the above supremum is attained \cite{DOV18} by a path $\gamma_{(x,s)}^{(y,t)}$, and such a path is called a geodesic from $(x,s)$ to $(y,t)$. Further, for any fixed $(x,s;y,t)\in \RR_\uparrow^4$, there is an a.s.\ unique geodesic $\gamma_{(x,s)}^{(y,t)}$.

In contrast to what happens in Euclidean geometry, the geodesics $\gamma_{(x,s)}^{(y,t)}$ are highly irregular objects. For example, for any fixed $(x,s;y,t)\in \RR_\uparrow^4$, the geodesic $\gamma_{(x,s)}^{(y,t)}$, viewed as a real valued function on $[s,t]$ is a.s.\ H\"older $2/3-$ regular \cite{DOV18} but not H\"older $2/3$ regular \cite{DSV22}, and the set of times where $\gamma_{(0,0)}^{(0,1)}$ intersects the line of points having space coordinate zero, a.s.\ has Hausdorff dimension $1/3$ \cite{GZ22}. Importantly, geodesics in the directed landscape enjoy the property of coalescence \cite{BSS19,GZ22,RV21}, wherein geodesics tend to merge with each other.

Apart from finite geodesics, one can even define \cite{BSS22} (see also \cite{RV21,GZ22}) semi-infinite geodesics in the directed landscape. Indeed, for any $\theta\in \RR$ and $p=(y,t)\in \RR^2$, a path $\Gamma_p^\theta\colon (-\infty,t]\rightarrow \RR$ is said to be a downward $\theta$-directed semi-infinite geodesic emanating from $p$ if every finite segment of $\Gamma_p^\theta$ is a geodesic, and further, $\Gamma_p^\theta(-s)/s\rightarrow \theta$ as $s\rightarrow \infty$. It can be shown that such geodesics almost surely exist simultaneously for all $\theta\in \RR$ and $ p\in \RR^2$, and for any fixed value of $\theta,p$, there is an a.s.\ unique geodesic $\Gamma^\theta_p$. In fact, it turns out that for any all $\theta$ lying outside a random countable set $\Xi_{\downarrow}$ \cite{BSS22}, the set $\cT_\downarrow^\theta=\bigcup_{p\in \RR^2}\inte(\Gamma_p^\theta)$ forms a tree \cite{Bha23}, where the union above is over all possible geodesics $\Gamma_p$ and $\inte(\Gamma_p^\theta)$ refers to the graph of $\Gamma_p^\theta$ but with the endpoint $p$ removed. Further, for $\theta\in \Xi_\downarrow^c$, the tree $\cT_\downarrow^\theta$ comes interlocked with a dual tree $\cI_\uparrow^\theta$ consisting of interfaces which go upwards, and we shall use $\{\Upsilon_p^\theta\}_{p\in \RR^2}$ to denote such interfaces, where we note that a point $p\in \RR^2$ might have multiple interfaces emanating from it, but for any fixed $p,\theta$, the interface $\Upsilon_p^\theta$ is a.s.\ unique. For more details regarding the objects $\cT_\downarrow^\theta, \cI_\uparrow^\theta$ from \cite{Bha23}, we refer the reader to Section \ref{sec:interf}. In this paper, we shall often work with $\theta=0$, and in this case, we shall often omit it from the notation; for instance, $\Gamma_p$ shall refer to $\Gamma_p^0$.

Among models of random geometry, a very unique and mysterious property of the directed landscape is its signed nature, that is, $\cL(x,s;y,t)$ can take both positive and negative values. As a result, for any $(x,s;y,t)\in \RR_\uparrow^4$, if we write $\cL(x,s;y,t)$ as the sum of lengths of small segments of the geodesic $\gamma_{(x,s)}^{(y,t)}$, there is a non-trivial cancellation occurring in this sum. Indeed, it can be shown \cite{DSV22} that the length function $s\mapsto \cL(0,0;\gamma_{(0,0)}^{(0,1)}(s))$ is a.s.\ H\"older $1/3-$ and not H\"older $1/3$, and thus if one were to naively add up the absolute values of the $\varepsilon^{-1}$ many segments of time length $\varepsilon$ on a geodesic, then this would yield a quantity diverging as $\varepsilon^{-2/3+o(1)}$, thereby pointing towards the intricacy of the cancellation structure present in the picture.%

Due to the non-triviality of the above-mentioned cancellation, it is for instance  possible that for the geodesic $\gamma_{(0,0)}^{(0,1)}$, there are subsets $S\subseteq [0,1]$ of measure zero such that points $(\gamma_{(0,0)}^{(0,1)}(s),s)$ for $s\in S$ still have a non-zero contribution to the length of a geodesic. The following question, which I heard for the first time from B{\'a}lint Vir{\'a}g at the IAS is related to controlling the length accumulated by a geodesic at the points where it intersects a vertical line.

\begin{question}
  \label{ques:vert}
  Given the lengths of all paths $\gamma\colon [0,1]\rightarrow \RR$ which do not touch the line $\{(0,s):s\in [0,1]\}$, can one obtain the values $\cL(p;q)$ for all $p=(x,s),q=(y,t)$ with $0<s<t<1$?
\end{question}

We now note that the above question can in fact be placed in a broader context, and we now define such a general notion.
\begin{definition}[informal]
  \label{def:metric}
  For $-\infty\leq s_0<t_0\leq \infty$, we say that a path $\xi\colon [s_0,t_0]\rightarrow \RR$ is metrically removable, if any $\cL(x,s;y,t)$ can be recovered given only the lengths of paths which never intersect the curve $\xi$.
\end{definition}

On doing a literature search, it appears that a similar but somewhat different notion with the same name as above is defined in the work \cite{KKR19}. However, in the present setting of two dimensional random geometry, where the focus is more on random curves, the above definition seems more appropriate. We are now ready to state the first result of this paper.
\begin{theorem}
  \label{thm:1}
  Almost surely, the interface $\Upsilon_{(0,0)}$ is metrically removable.
  
\end{theorem}
That is, while we do not investigate metric removability of the vertical line as asked in Question \ref{ques:vert}, we establish metric removability for the curve $\Upsilon_{(0,0)}$, which is a random curve determined by the directed landscape. One might wonder whether it is possible that any reasonable curve is metrically removable. However, this is not true, as we record in the following statement.

\begin{theorem}
  \label{thm:2}
  Almost surely, the geodesic $\gamma_{(0,0)}^{(0,1)}$ is not metrically removable.
\end{theorem}

Given the detailed study of disjoint optimizers \cite{DZ21} in the directed landscape, Theorem \ref{thm:2} is not difficult to obtain, and thus the focus of the paper is the proof of Theorem \ref{thm:1}. A major step in the proof of Theorem \ref{thm:1} will be to develop good Hausdorff dimension bounds on the intersection $\Upsilon_{(0,0)}$ and a geodesic $\gamma_p^q$. In fact, we shall go further, and pin down the precise value of the above-mentioned dimension. We now state this independently interesting result on the fractal geometry of the directed landscape.

\begin{theorem}
  \label{thm:3}
 Almost surely, for all $(p;q)\in \RR_\uparrow^4$, the set of times $s$ for which $\gamma_p^q(s)=\Upsilon_{(0,0)}(s)$ has Hausdorff dimension zero. Further, if we consider the metric $d_{\mathrm{KPZ}}$ on $\RR^2$ defined by $d_{\mathrm{KPZ}}(x,s;y,t)=|x-y|^{1/2}+|s-t|^{1/3}$ and define the geodesic frame $\cW$ as the union of interiors of all geodesics in the directed landscape, then almost surely, $\cI_\uparrow\cap \cW$ is an infinite set whose Hausdorff dimension with respect to $d_{\mathrm{KPZ}}$ is equal to $0$.
\end{theorem}
In the past few years, there has been significant effort to understand the fractal structure of the directed landscape, and the above result is in the same vein. Some of the noteworthy themes that have been explored recently are the study of atypical stars and geodesic networks \cite{BGH21, BGH19, GZ22, Bha22, Ham20, Dau23+}, the study of semi-infinite geodesics across all directions \cite{BSS22,Bha23,Bus23} and the investigation of exceptional times for the KPZ fixed point \cite{CHHM21,Dau22}. Also, we note that the deterministic metric $d_{\mathrm{KPZ}}$ in Theorem \ref{thm:3} has been employed earlier in the literature, namely in the works \cite{CHHM21,Dau23+,BB23}. For the present setting, the most relevant work is \cite{Bha22}, where the dimension of atypical stars, or the set of points of failure of coalescence, is computed along a geodesic between fixed points. As we shall see later, this will play a major role in the proof of Theorem \ref{thm:3}.

The notion of metric removability in Definition \ref{def:metric} is in fact motivated from Liouville quantum gravity ($\gamma$-LQG) \cite{She22,DDG21}-- a prominent one parameter family of models of random geometry, indexed by $\gamma\in (0,2]$, and believed to arise as the scaling limits of a variety of discrete planar map models. These geometries come with a conformal structure attached to them and have an intricate theory where one can  ``cut'' a $\gamma$-LQG surface \cite{She16,DMS14,AHS24} along an independent Schramm Loewner evolution ($\mathrm{SLE}_{\gamma^2}$) curve to yield $\gamma$-LQG surfaces on its two sides which are conditionally independent given the boundary length of the curve. In this context, a natural heavily-investigated question which arises and is parallel to Definition \ref{def:metric}, is that of conformal removability, where a curve is said to be conformally removable if the conformal structure on its two sides determines the overall conformal structure of the surface. It is known that $\mathrm{SLE}_{\gamma^2}$ curves are conformally removable \cite{JS00,RS05,KMS22} for all $\gamma\in (0,2]$, and in fact, also that geodesics in $\gamma$-LQG are conformally removable \cite{MQ20+}. In fact, a closer parallel to Definition \ref{def:metric} is given by the work \cite{HM22}, where the authors construct $\gamma$-LQG as a ``metric gluing'' of smaller $\gamma$-LQG surfaces, or in the language of Definition \ref{def:metric}, they show that appropriate $\mathrm{SLE}_{\gamma^2}$-type curves are metrically removable in $\gamma$-LQG.

\subsection{A precise formulation of metric removability}
\label{sec:restr}
Due to the informal nature of Definition \ref{def:metric}, we now give a precise formulation of metric removability for the directed landscape. We begin by defining the restriction of the directed landscape to an open set $U\subseteq \RR^2$.
  \begin{figure}
  \begin{subfigure}[h]{0.4\linewidth}
    \centering
\includegraphics[width=0.6\linewidth]{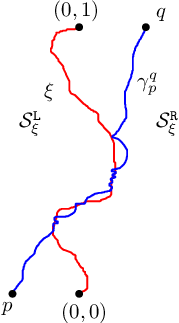}
\end{subfigure}
\hfill
\begin{subfigure}[h]{0.4\linewidth}
  \centering
\includegraphics[width=0.6\linewidth]{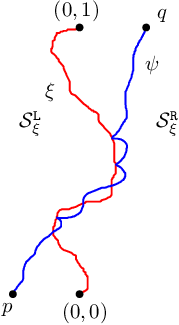}
\end{subfigure}
\caption{\textbf{Left panel}: The geodesic $\gamma_p^q$ here intersects the path $\xi$ uncountably many times. \textbf{Right panel}: Here, for a given $\varepsilon>0$, $\psi$ is a path from $p$ to $q$ intersecting $\xi$ finitely many times but still satisfying $\ell(\psi;\cL)\geq \cL(p;q)-\varepsilon$.}
\label{fig:finhit}
\end{figure}
\begin{definition}
  Given an open set $U\subseteq \RR^2$, we define the restricted landscape $\cL_U$ as follows. For any $p=(x,s), q=(y,t)\in \overline{U}$ with $s<t$, we define
  \begin{equation}
    \label{eq:1}
    \cL_U(p;q)=\sup_{\gamma\colon [s,t]\rightarrow \RR, \inte(\gamma)\subseteq U}\ell(\gamma;\cL),
  \end{equation}
  with the convention that $\cL_U(p;q)=-\infty$ if no path $\gamma$ as in the above exists.
\end{definition}

With some effort, it is possible to verify the measurability of $\cL_U(p;q)$ for any fixed $p,q$ as above, and thus $\cL_U(p;q)$ is a random variable. We can now restrict the directed landscape to the left and right of a given curve $\xi$.
\begin{definition}
  \label{def:side}
    For $-\infty\leq s'_o<t'_0\leq \infty$, and a path $\xi\colon [s'_0,t'_0]\rightarrow \RR$, we define the sets to the left and right of $\xi$ by
  \begin{displaymath}
    \cS^{\tL}_{\xi}=\{(x,r): r\in [s'_0,t'_0], x< \xi(r)\}, \cS^{\tR}_{\xi}=\{(x,r): r\in [s'_0,t'_0], x>  \xi(r)\}.
  \end{displaymath}
 Further, we define $\cL^{\tL}_{\xi}$ and $\cL^{\tR}_{\xi}$ by
  \begin{displaymath}
    \cL^{\tL}_{\xi}=\cL_{\cS^{\tL}_{\xi}},  \cL^{\tR}_{\xi}=\cL_{\cS^{\tR}_{\xi}}.
  \end{displaymath}
\end{definition}
Given a path $\xi$ and the restricted landscapes $\cL_\xi^{\tL},\cL_\xi^{\tR}$, we now define the reconstruction $\cL_\xi^g$ obtained by ``metrically gluing'' $\cL_\xi^{\tL}$ and $\cL_\xi^{\tR}$.
\begin{definition}
  \label{def:recons}
  Given a path $\xi$ as in the above, we define the reconstruction $\cL^{\g}_{\xi}$ such that for any $(p;q)=(x,s;y,t)\in \RR_\uparrow^4$ satisfying  $s'_0\leq s<t\leq t'_0$, $\cL^{\g}_{\xi}(p;q)$ is equal to
  \begin{displaymath}
 \sup_{s<r_1<\dots <r_n<t} \left\{(\cL^{\tL}_{\xi}\vee \cL^{\tR}_{\xi})(p;\xi(r_1),r_1)+\sum_{i=1}^{n-1} (\cL^{\tL}_{\xi}\vee \cL^{\tR}_{\xi})(\xi(r_i),r_i;\xi(r_{i+1}),r_{i+1})+(\cL^{\tL}_{\xi}\vee \cL^{\tR}_{\xi})(\xi(r_n),r_n;q)\right\}.
  \end{displaymath}
  where the supremum is over all partitions $s=r_0< \dots < r_n=t$, and the $\vee$ above simply denotes the maximum.
\end{definition}
In other words, $\cL_\xi^g(p;q)$ is the supremum of $\ell(\psi;\cL)$ over paths $\psi$ from $p$ to $q$ which intersect $\xi$ only finitely many times (see Figure \ref{fig:finhit}). We note that a definition analogous to the above, but in the setting of $\gamma$-LQG appears as \cite[Definition 1]{HM22}. Now, in Definition \ref{def:recons}, it is easy to see by that we always have $\cL_{\xi}^{\g}(p;q)\leq \cL(p;q)$.  We now formulate a precise definition of metric removability, and we take care to ensure that it is defined in terms of countably many operations, so as to not run into measure theoretic difficulties.

\begin{definition}
  \label{def:rem}
  Fix $-\infty\leq s'_0<t'_0\leq \infty$. A possibly random path $\xi\colon [s'_0,t'_0]\rightarrow \RR$ is said to be metrically removable if for any fixed $p=(x_0,s_0)$ and $q=(y_0,t_0)$ with $s_0'<s_0<t_0<t_0'$, we almost surely have $\cL^g_\xi(p;q)=\cL(p;q)$.
\end{definition}
We note that the results of this paper on metric removability (Theorems \ref{thm:1}, \ref{thm:2}) are proved using the above-mentioned definition. In particular, for the rest of the paper, we shall work with the above definition of metric removability.

\subsection{A discussion of upcoming work}
This work is part of a series of three works aimed at solving a certain reconstruction problem in the directed landscape, namely, to show that the directed landscape $\cL$ is determined by the parametrised geodesic tree $\cT_\downarrow$, by which we mean the set $\cT_\downarrow$ along with the lengths of all geodesics contained in it. These works shall strongly use the Peano curve $\eta\colon \RR\rightarrow \RR^2$ snaking between the trees $\cT_\downarrow$ $\cI_\uparrow$ as constructed in the recent work \cite{BB23}. We now give a short description of the above-mentioned upcoming works.

\begin{enumerate}
\item In the upcoming work \cite{Bha24+}, we shall obtain a conditional independence result for the bubbles created by the Peano curve $\eta$. Parametrising $\eta$ by the Lebesgue volume and such that $\eta(0)=0$, we shall show that for any fixed $v_1<\dots <v_n\in \RR$, the objects $\cL_{\eta[v_i,v_{i+1}]}$ defined as the directed landscape restricted to the Peano bubbles $\eta[v_i,v_{i+1}]$, are mutually independent conditional on the Busemann function $\cB_0$ (see Section \ref{sec:busemann}) restricted to the union $\bigcup_{i=1}^n\partial_{v_i}\eta$, where $\partial_{v_i}\eta$ is the topological boundary of the set $\eta(-\infty,v_i]$.

\item Subsequently, in another upcoming work \cite{Bha24++}, we shall show that the conditional variance of any directed landscape length $\cL(p;q)$ given the parametrised geodesic tree $\cT_\downarrow$ is almost surely equal to zero, thereby establishing that the directed landscape can a.s.\ be reconstructed from the parametrised geodesic tree. Apart from using Theorem \ref{thm:1} from the present work along with the conditional independence from the above-mentioned upcoming work \cite{Bha24+}, this shall crucially require results from the work \cite{Bha23} on atypical stars on geodesics and tail estimates from \cite{BB23} on the size of Peano bubbles.
\end{enumerate}
\paragraph{\textbf{Notational comments}} For $a<b\in \RR$, we shall use $[\![a,b]\!]$ to denote the set $[a,b]\cap \ZZ$. Throughout the paper, we shall use $\0$ to denote the point $(0,0)\in \RR^2$ and shall also set $\varepsilon_n=2^{-n}$ for $n\in \NN$. Also, as usual, if we have a function which is $\alpha-\varepsilon$ H\"older continuous for all $\varepsilon>0$, then we shall simply say that the function is $\alpha-$ H\"older regular. As remarked earlier, we shall refer to a continuous function $\eta\colon [s,t]\rightarrow \RR$ as a path and shall call it a semi-infinite path in case either of $s,t$ is equal to $\pm \infty$. For any such path $\xi$, we shall use $\inte(\xi)$ to denote the interior of the path, by which we simply mean the graph of the path $\xi$ with its endpoints removed. As an overload of notation, for a path $\xi$, we shall often use $\xi$ to also denote its graph; for example, if we write $p\in \xi$, then we mean that the point $p$ lies on the graph of $\xi$. Also, we shall often have paths $\xi_1,\xi_2$ which potentially intersect each other only at their endpoints, and in this case, we shall call these paths \emph{almost disjoint}. Throughout the paper, we shall often work with left-most and right-most geodesics (see Proposition \ref{prop:21}), and in this case, we shall underline (resp.\ overline) to denote the left-most (resp.\ right-most geodesic). For example, $\underline{\gamma}_p^q$ refers to the left-most geodesic corresponding to some $(p;q)\in \RR_\uparrow^4$. Finally, in Section \ref{sec:proof-thm-dim}, we shall use $\wcL$ to denote the directed landscape obtained by independently resampling the portion of $\cL$ below the time line $0$.
\paragraph{\textbf{Acknowledgements}} The author thanks B{\'a}lint Vir{\'a}g for the discussion at the IAS in Spring 2023 and also thanks Duncan Dauvergne for the discussion at the Fields institute. The author acknowledges the partial support of the NSF grant DMS-2153742 and also acknowledges the Fields Institute for their hospitality since the work was completed while attending a program there on randomness and geometry. 
\label{sec:broad}

\section{A discussion of the proof strategy}
\label{sec:outline}
We now give an outline of the proof of Theorem \ref{thm:1}. Given a random curve $\xi$ coupled to the directed landscape which one believes to be metrically removable, a possible strategy to demonstrate this is the following.

\begin{enumerate}
\item For fixed rational points $p=(x_0,s_0)$ and $q=(y_0,t_0)$ with $s_0<t_0$, obtain a good a.s.\ upper bound on the Minkowski dimension of the set $\{s\colon \gamma_p^q(s)=\xi(s)\}$. Say we obtain that this dimension is upper bounded by $d$ almost surely.
\item For each point $(\xi(s),s)$ on $\xi$, exhibit finite length paths $\pi^{\tL}_s,\pi^{\tR}_s$ (say emanating downwards from $(\xi(s),s)$) which satisfy $\inte(\pi^{\tL}_s)\subseteq \cS_\xi^{\tL}$ and $\inte(\pi^{\tR}_s)\subseteq \cS_\xi^{\tR}$. Further, we wish to have good simultaneous control (for all $s$ in a compact set) on the transversal fluctuation behaviour of these paths and good bounds on $\ell(\pi^{\tL}_s\lvert_{[s-\delta,s]};\cL),\ell(\pi^{\tR}_s\lvert_{[s-\delta,s]};\cL)$ for all $\delta>0$.
\end{enumerate}
Given a setting resembling the above, one could hope to show metric removability via a local modification argument. Indeed, we divide the interval $[s_0,t_0]$ into sub-intervals of length $\varepsilon_n=2^{-n}$ each and first use (1) to obtain that almost surely, the number of bad intervals, by which we mean those intervals $I$ which contain a point $s$ for which $\gamma_p^q(s)=\xi(s)$, is a.s.\ at most $\varepsilon_n^{-d-o_n(1)}$. Having done so, the strategy is to modify $\gamma_p^q$ only around the bad intervals to create a sequence of paths $\gamma_n$ from $p$ to $q$ with the property that for each $n$, the set $\{s\colon \gamma_n(s)=\xi(s)\}$ is finite, and further, $\ell(\gamma_n;\cL)\rightarrow \ell(\gamma_p^q;\cL)$ as $n\rightarrow \infty$. By the definition of metric removability, it is not hard to see that this would imply $\cL_\xi^g(p;q)=\cL(p;q)$.

\subsubsection*{\textbf{Constructing the paths $\gamma_n$}} We now present a way to construct the desired paths $\gamma_n$, and for this construction, we wish to use the paths $\pi_s^{\tL},\pi_s^{\tR}$ available to us. Now, for each bad interval $I=[r,r+\varepsilon_n]$, we consider the points $\underline{s}_I=\inf\{s\in I\colon \gamma_p^q(s)=\xi(s)\}$ and $\overline{s}_I=\inf\{s\in I\colon \gamma_p^q(s)=\xi(s)\}$, and we note that we potentially have infinitely many intersections of $\gamma_p^q$ and $\xi$ in the time interval $[\underline{s}_I,\overline{s}_I]$, and the aim now is to avoid these intersections by using the paths $\pi^{\tL}_{\overline{s}_I},\pi^{\tR}_{\overline{s}_I}$. Let us consider the generic case when $r<\underline{s}_{I}<\overline{s}_I<r+\varepsilon_n$. In this case, there exists a unique $*_I\in \{\tL,\tR\}$ defined such that $(\gamma_p^q(r),r)\in \cS_\xi^{*_I}$, where we recall that the latter was defined in Definition \ref{def:side}. A possible strategy now is to modify $\gamma_p^q$ by replacing $\gamma_p^q\lvert_{[r,\overline{s}_I]}$ by the concatenation of the curves $\gamma_{(\gamma_p^q(r),r)}^{(\pi^{*_I}_{\overline{s}_I}(\underline{s}_I),\underline{s}_I)}$,    $\pi^{*_I}_{\overline{s}_I}\lvert_{[\underline{s}_I,\overline{s}_I]}$ (see Figure \ref{fig:mod} for a depiction of this step in the actual construction). %

Doing this would ensure the modified path has at most one intersection with $\xi$ in the interval $I$, namely at $\overline{s}_I$. Further, if the length estimate in (2) holds, then one could hope (see Proposition \ref{prop:11}) to have a good upper bound on $|\ell(\pi^{*_I}_{\overline{s}_I}\lvert_{[\underline{s}_I,\overline{s}_I]};\cL)|$. Finally, if the transversal fluctuation bound in (2) holds and if $\underline{s}_I$ is not too close to $r$, then one could also hope to achieve a similar control on $\ell(\gamma_{(\gamma_p^q(r),r)}^{(\pi^{*_I}_{\overline{s}_I}(\underline{s}_I),\underline{s}_I)};\cL)$. If the dimension upper bound $d$ is small enough, then one could hope that even on adding up, the absolute values of the above errors over all bad intervals, we obtain a quantity which still decays to $0$ in $n$, thereby yielding that $\ell(\gamma_n;\cL)\rightarrow \ell(\gamma_p^q;\cL)$.

The above strategy does not seem to work directly in the case when $\xi$ is the deterministic curve given by $\xi(s)=0$ for $s\in [0,1]$ (see Question \ref{ques:vert}), since it yields $d=1/3$, which is not good enough. However, if we instead work with the curve $\xi=\Upsilon_{\0}$, then the above strategy does in fact work, and we now outline why this is the case. To do so, we just explain why (1) and (2) should hold for the interface $\Upsilon_{\0}$. We first begin with condition (2).

\subsubsection*{\textbf{The condition (2) for the interface $\Upsilon_{\0}$}} In some sense, the interface $\Upsilon_{\0}$ is ideal for checking condition (2). Indeed, if we use $\underline{\Gamma}_p,\overline{\Gamma}_p$ to denote the left-most and right-most choices of $\Gamma_p$ (see Proposition \ref{prop:21}), then by \cite{RV21,Bha23}, we know that almost surely, for all $s>0$,
\begin{equation}
  \label{eq:88}
 \underline{\Gamma}_{(\Upsilon_{\0}(s),s)}\lvert_{[0,s)}\subseteq \cS_{\Upsilon_{\0}}^{\tL}, \overline{\Gamma}_{(\Upsilon_{\0}(s),s)}\lvert_{[0,s)}\subseteq \cS_{\Upsilon_{\0}}^{\tR},
\end{equation}
and this leads to a natural choice of the paths $\pi^{\tL}_s$ and $\pi^{\tR}_s$. Note that here, $\pi^{\tL}_s$ and $\pi^{\tR}_s$ are genuine $\cL$-geodesics as opposed to just being paths of finite length, and as a result, it is easy to obtain the required transversal fluctuation and length control described in (2), by simply employing corresponding uniform bounds for $\cL$-geodesics (see Propositions \ref{prop:7}, \ref{prop:17}).

\subsubsection*{\textbf{The condition (1) for the interface $\Upsilon_{\0}$}} Now, with $\xi=\Upsilon_{\0}$, we wish to bound the Minkowski dimension of the set $\{s\colon \gamma_p^q(s)=\xi_{\0}(s)\}$. It turns out that this dimension can be precisely computed-- it is almost surely equal to zero. We note that apart from its usage in proving the metric removability of $\Upsilon_{\0}$, the above dimension computation will be used to yield Theorem \ref{thm:3}.

Now, instead of discussing condition (1) for the path $\xi =\Upsilon_{0}$, we shall discuss the case when $\xi=\LL_0$, where the path $\LL_0$ is defined such that
\begin{equation}
  \label{eq:122}
  \LL_0=\{(y,t): t>0, \cL(1,0;y,t)-\cL(-1,0;y,t)=0\}.
\end{equation}
It can be shown that $\LL_0$ is a.s.\ a semi-infinite path emanating upward and has the property that all points on this path admit almost disjoint geodesics going to $(-1,0)$ and $(1,0)$. We note that in the sense of \cite{RV21} (see Section \ref{sec:interf}), $\LL_0$ can be viewed as the interface from $\0$ for the initial condition given by $\fh(x)=0$ if $|x|=\pm 1$ and $-\infty$ otherwise. Analogously, $\Upsilon_{\0}$ can be interpreted as the initial condition given by a two-sided Brownian motion (given by the Busemann function on the line with time coordinate $0$, see Section \ref{sec:interf}). Thus, both $\LL_0$ and $\Upsilon_{\0}$ can be interpreted as interfaces starting from certain initial conditions. To effectively communicate the intuitive picture, we describe how to verify condition (1) for $\LL_0$ in this section. The actual argument works with $\Upsilon_{\0}$ and uses the same intuition but is more technically involved.

By standard arguments employing transversal fluctuation estimates, it is not difficult to see that it would suffice to show that for any fixed $s\in [s_0,t_0]$ (actually, one needs some uniformity over $s$, but for this section, let us work with a fixed $s$), any fixed $\nu>0$, and for all small enough $\delta>0$, we have
\begin{equation}
  \label{eq:133}
  \PP(|\gamma_p^q(s)-\LL_0(s)|\leq \delta)\leq \delta^{3/2-\nu}.
\end{equation}
Indeed, one could then choose $\delta=\varepsilon^{2/3-o(1)}$ and cover $[s_0,t_0]$ by intervals of size $\varepsilon$. By using the $2/3-$ H\"older nature of $\gamma_p^q$ and $\LL_0$, \eqref{eq:133} would imply that each such interval has at most a $\varepsilon^{1-o(1)}$ probability of containing a time when $\gamma_p^q$ and $\LL_0$ intersect, and this would imply that the Minkowski dimension of such intersection times is a.s.\ zero.

\subsubsection*{\textbf{The estimate \eqref{eq:133}}} Let $\cD_s(x)$ denote the difference profile defined by $\cD_s(x)=\cL(1,0;x,s)-\cL(-1,0;x,s)$, and it can be checked that $\cD_s$ is increasing in the sense that almost surely, for any $y\geq x$, we have $\cD_s(y)\geq \cD_s(x)$. Owing to the $1/2-$ spatial H\"older continuity of $\cL$, $\cD_s$ is also $1/2-$ H\"older continuous, and in fact, by using basic regularity and transversal fluctuation estimates, it is not difficult to show that, for any fixed $\nu>0$, if we define the event $E_\delta$ by
\begin{equation}
  \label{eq:135}
  E_\delta=\{ \cD_s(\gamma_p^q(s)+\delta)-\cD_s(\gamma_p^q(s)-\delta)\geq \delta^{1/2-\nu}\},
\end{equation}
then $\PP(E_\delta)$ decays superpolynomially as $\delta\rightarrow 0$. As a result, it can be checked that we have the following useful inclusion for the event from \eqref{eq:133},
\begin{align}
  \label{eq:134}
  E_\delta^c\cap \{|\gamma_p^q(s)-\LL_0(s)|\leq \delta\}&\subseteq  \{\exists x\in [\gamma_p^q(s)-\delta,\gamma_p^q(s)+\delta]: \overline{\gamma}^{(x,s)}_{(-1,0)} \textrm{ and } \overline{\gamma}^{(x,s)}_{(1,0)} \textrm{ are almost disjoint}\}\nonumber\\
                                                        &\cap \{ |\cD_s(\gamma_p^q(s))-0|\leq \delta^{1/2-\nu}\}\nonumber\\
  &\coloneqq F_\delta\cap G_\delta,
\end{align}
and thus it suffices to obtain a $\delta^{3/2-\nu}$ probability upper bound on the event above. Now, the work \cite{Bha22} considers a version of the event
\begin{equation}
  \label{eq:177}
  \{\exists x\in [\gamma_p^q(s)-\delta,\gamma_p^q(s)+\delta] \textrm{ admitting almost disjoint geodesics } \gamma^{(x,s)}_p\textrm{ and }  \gamma^{(x,s)}_{(1,0)}\}
\end{equation}
for semi-infinite geodesics, for which it proves a $\delta^{1-o(1)}$ probability upper bound. It can be shown that if the event $F_\delta$ occurs then we can either find almost disjoint geodesics $\gamma_p^{(x,s)},\gamma_{(1,0)}^{(x,s)}$ or almost disjoint geodesics $\gamma_p^{(x,s)},\gamma_{(-1,0)}^{(x,s)}$, and this suggests that $\PP(F_\delta)\leq\delta^{1-o(1)}$. To complete the proof, one must now show that
\begin{equation}
  \label{eq:2}
  \PP(G_\delta\lvert F_\delta)\leq \delta^{1/2-o(1)},
\end{equation}
 and intuitively, this would amount to establishing that $\cD_s$ has a density around $0$ even when we \emph{condition} on $F_\delta$. Overall, the events $F_\delta$ and $G_\delta$ involve information about geodesics from three different points: $(-1,0)$, $p$ and $(1,0)$, and as a result, it does not seem easy to directly prove \eqref{eq:2}. Instead of attempting to show \eqref{eq:2} directly and thus tackling the above non-trivial conditioning headfirst, we instead sidestep this difficulty by using the skew invariance of the landscape along with additional randomness independent of $\cL$.

\subsubsection*{\textbf{The random variable $X$ and the interfaces $\LL_X$}} Now, for $r\in \RR$, we define the $r$-difference set
\begin{equation}
  \label{eq:136}
  \LL_r=\{(y,t): t>0, \cD_t(y)=r\}.
\end{equation}
Now, we note that the directed landscape satisfies a certain skew invariance as summarised in Proposition \ref{prop:14}, and we shall use the landscape $\cL_\theta^{\mathrm{sk}}$ defined for $\theta\in \RR$ therein. Now, it can be checked that the set $\LL_0$ for $\cL$ skew-transforms to the $-4\theta$-difference set for $\cL_\theta^{\mathrm{sk}}$. As a result of this, it can be seen that to obtain condition (1) for $\LL_0$, we can equivalently fix an $M>0$, sample $X\sim \mathrm{Unif}[-M,M]$ independently of $\cL$ and establish that condition (1) is a.s.\ satisfied for $\LL_X$ instead of $\LL_0$.

As a result of the above, it suffices to show that the set of times where $\gamma_p^q$ and $\LL_X$ intersect a.s.\ has Minkowski dimension $0$. To do so, we must establish an estimate analogous to \eqref{eq:133}, where $\LL_0$ is now replaced by $\LL_X$. Now, just as in \eqref{eq:134}, we can write
\begin{align}
  \label{eq:138}
    E_\delta^c\cap \{|\gamma_p^q(s)-\LL_X(s)|\leq \delta\}&\subseteq  \{\exists x\in [\gamma_p^q(s)-\delta,\gamma_p^q(s)+\delta]: \overline{\gamma}^{(x,s)}_{(-1,0)} \textrm{ and } \overline{\gamma}^{(x,s)}_{(1,0)} \textrm{ are almost disjoint}\}\nonumber\\
                                                        &\cap \{ |\cD_s(\gamma_p^q(s))-X|\leq \delta^{1/2-\nu}\}\nonumber\\
  &\coloneqq  F_\delta'\cap G_\delta',
\end{align}
for the regularity event $E_\delta$ from \eqref{eq:135} whose probability decays superpolynomially as $\delta\rightarrow 0$. However, this time, we can in fact obtain the desired $\delta^{3/2-o(1)}$ estimate on $\PP(F_\delta'\cap G_\delta')$. Indeed, as earlier, by using the results from \cite{Bha22}, we can obtain $\PP(F_\delta')\leq \delta^{1-o(1)}$. However, now we can write
\begin{equation}
  \label{eq:139}
  G_\delta'=\{X\in [\cD_s(\gamma_p^q(s))-\delta^{1/2-\nu},\cD_s(\gamma_p^q(s))+\delta^{1/2-\nu}]\},
\end{equation}
and due to the independence of $X$ and $\cL$, on the regularity event $E_\delta^c$, we have
\begin{equation}
  \label{eq:140}
  \PP(G_\delta'\lvert \cL)\leq 2\delta^{1/2-\nu}/(2M).
\end{equation}
Combined with the $\delta^{1-o(1)}$ bound on $\PP(F_\delta')$, this implies the desired $\delta^{3/2-o(1)}$ bound on $\PP(F'_\delta\cap G'_\delta)$, thereby showing that the Minkowski dimension of the set of $s$ for which $\gamma_p^q(s)=\LL_X(s)$ should a.s.\ be equal to zero.

\subsubsection*{\textbf{Some comments on the actual proof}}
In the actual proof, we do not work with any of the sets $\LL_0$ or $\LL_X$. Instead, when checking condition (1), we work with interfaces starting from a Brownian motion with a drift $2X$ independent of $\cL$. The argument here is more technical than the one presented above for $\LL_0$ and $\LL_X$ and, for instance, involves first working with initial conditions $\fh_X^K$ obtained by truncating the above initial condition at a level given by a parameter $K$. Also, there are some additional steps required to handle the fact that the results from \cite{Bha22} only concern semi-infinite geodesics. However, though there are the above-mentioned additional technicalities, the argument on the whole is along the same lines as the one outlined in this section.

\section{Preliminaries}
\label{sec:prelim}

\subsection{Some relevant properties of the directed landscape}
\label{sec:geod}
Throughout this paper, we shall require certain properties of the directed landscape and its geodesics which we now introduce. We begin by introducing the basic symmetries enjoyed by the directed landscape.
\begin{proposition}[{\cite[Lemma 10.2]{DOV18}, \cite[Proposition 1.23]{DV21}}]
  \label{prop:14}
  As a random continuous function from $\RR_\uparrow^4$ to $\RR$, the following distributional equalities hold for the directed landscape:
   \begin{itemize}
   \item KPZ $1\colon 2 \colon 3$ scaling: For any $q>0$, $\cL(x,s;y,t)\stackrel{d}{=}q\cL(q^{-2}x,q^{-3}s;q^{-2}y,q^{-3}t)$.
   \item Skew invariance: For any $\theta\in \RR$, if we define $\cL_\theta^{\mathrm{sk}}\colon \RR_\uparrow^4\rightarrow \RR$ by
     \begin{displaymath}
       \cL_\theta^{\mathrm{sk}}(x,s;y,t)= \cL(x+\theta s,s;y+\theta t, t) +\theta^2(t-s) +2\theta(y-x),
     \end{displaymath}
   then $\cL_\theta^{\mathrm{sk}}\stackrel{d}{=}\cL$.
  \item Translation invariance: For any $x_0,s_0\in \RR$, $\cL(x,s;y,t)\stackrel{d}{=}\cL(x+x_0,s+s_0;y+x_0,t+s_0)$.
  \item Flip invariance: $\cL(x,s;y,t)\stackrel{d}{=}\cL(-x,s;-y,t)\stackrel{d}{=}\cL(y,-t;x,-s)$.
  \end{itemize} 
\end{proposition}

Also, we have the following uniform estimate from \cite{DOV18} on the values  $\cL(x,s;y,t)$.
\begin{proposition}[{\cite[Corollary 10.7]{DOV18}}]
  \label{prop:7}
  There exists a random variable $N$ satisfying $\PP(N\geq m)\leq Ce^{-cm^{3/2}}$ for some constants $C,c$ and all $m>0$ such that for all $u=(x,s;y,t)\in \RR_\uparrow^4$, we have
    \begin{equation}
      \label{eq:22}
      |\cL(x,s;y,t)+(x-y)^2/(t-s)|\leq N (t-s)^{1/3} \log^{4/3}\left(\frac{2\|u\|+2}{s}\right)\log^{2/3}(\|u\|+2),
    \end{equation}
    where $\|u\|$ denotes the $L^2$ norm.
  \end{proposition}
  The above result shows that globally, the behaviour of $\cL(x,s;y,t)$ is governed by the parabolic term $(x-y)^2/(t-s)$. However, the local behaviour is Brownian, and we now record an estimate in this direction.
  \begin{proposition}[{\cite[Proposition 1.6]{DOV18}}]
  \label{lem:20}
  Fix $\nu>0$. Then the probability %
  \begin{equation}
    \label{eq:83}
    \PP(\sup_{x,x',y,y'\in [-\delta,\delta]}|\cL(x,0;y,1)-\cL(x',0;y',1)| \geq \delta^{1/2-\nu}) 
  \end{equation}
decays superpolynomially as $\delta\rightarrow 0$.
\end{proposition}
We note that \cite[Proposition 1.6]{DOV18} is more quantitative than the above result, but the above simpler result will be sufficient for us. We now turn to estimates involving geodesics in the directed landscape. As we noted in the introduction, the geodesic $\gamma_p^q$ is a.s.\ unique \cite[Theorem 12.1]{DOV18} for any fixed $(p;q)\in \RR_\uparrow^4$. Though there do exist \cite{Dau23+} exceptional pairs $(p;q)$ with multiple such geodesics, there always do exist unique left-most and right-most geodesics, and we record this in the following lemma.
\begin{proposition}[{\cite[Lemma 13.2]{DOV18}, \cite[Theorem 6.5 (i)]{BSS22}}]
  \label{prop:21}
  Almost surely, simultaneously for every $(x,s;y,t)\in \RR_\uparrow^4$, there exist a left-most geodesic $\underline{\gamma}_{(x,s)}^{(y,t)}$ and a right-most geodesic $\overline{\gamma}_{(x,s)}^{(y,t)}$ from $(x,s)$ to $(y,t)$ satisfying
  \begin{equation}
    \label{eq:159}
    \underline{\gamma}_{(x,s)}^{(y,t)}(r)\leq \gamma_{(x,s)}^{(y,t)}(r)\leq \overline{\gamma}_{(x,s)}^{(y,t)}(r)
  \end{equation}
  for all geodesics $\gamma_{(x,s)}^{(y,t)}$ and all $r\in [s,t]$. Similarly, almost surely, simultaneously for all $\theta\in \RR$ and $p=(y,t)\in \RR^2$, there exists a unique left-most downward $\theta$-directed semi-infinite geodesic $\underline{\Gamma}_p^\theta$ and a unique right-most downward $\theta$-directed semi-infinite geodesic $\overline{\Gamma}_p^\theta$ satisfying
  \begin{equation}
    \label{eq:160}
    \underline{\Gamma}_p^\theta(r)\leq \Gamma_p^\theta(r)\leq \overline{\Gamma}_p^\theta(r)
  \end{equation}
  for all geodesics $\Gamma_p^\theta$ and all $r\leq t$.
\end{proposition}
In \cite[Proposition 12.3]{DOV18}, it was shown that for fixed points $p,q$, the geodesic $\gamma_p^q$ is $2/3-$ H\"older regular. In fact, the above result for geodesics between fixed points is true simultaneously for all geodesics, and we now state a quantitative estimate which, in particular, implies this.
\begin{proposition}[{\cite[Lemma 3.11]{GZ22}}]
    \label{prop:17}
    There is a random variable $S$ and positive constants $C,c$ such that we have the following. For $\ell>0$, the tail estimate $\PP(S>\ell)<Ce^{-c\ell^{9/4}(\log \ell)^{-4}}$ holds. Further, for any $u\in (x,s;y,t)\in \RR_\uparrow^4$, any geodesic $\gamma_{(x,s)}^{(y,t)}$ and any $r$ satisfying $(s+t)/2\leq r<t$,
    \begin{equation}
      \label{eq:124}
      |\gamma_{(x,s)}^{(y,t)}-\frac{x(t-r)+y(r-s)}{t-s}|\leq S(t-r)^{2/3}\log^3(1+(t-r)^{-1}\|u\|),
    \end{equation}
    where $\|u\|$ denotes the $L^2$ norm. Also, by symmetry, a bound similar to the above holds when $s<r<(s+t)/2$.
  \end{proposition}
  Since the above holds for all geodesics, it is easy to see that it implies that the semi-infinite geodesic $\Gamma_{\0}$ must be a.s.\ locally H\"older $2/3-$ regular as well.
We shall also require the following transversal fluctuation estimate for infinite geodesics.
\begin{proposition}[{\cite[Theorem 3.12]{RV21}}]
  \label{lem:12}
  Fix $0<\ualpha<\oalpha$. Then there exist constants $C,c$ such that for any $\theta\in \RR$ and all $\ell>0$, we have
  \begin{equation}
    \label{eq:57}
    \PP(\sup_{t\in [\ualpha, \oalpha]}|\Gamma^\theta_q(-t)-\theta t|\geq \ell)\leq Ce^{-c\ell^3}.
  \end{equation}
\end{proposition}
Now, we discuss some useful results on the convergence properties of geodesics.
\begin{proposition}[{\cite[Lemma 3.1]{DSV22}}]
  \label{prop:18}
  The following holds with probability $1$. For all points $u=(x,s;y,t)\in \RR^4_\uparrow$ and any sequence of points $u_n=(x_n,s_n;y_n,t_n)\in \RR^4_\uparrow$ converging to $u$, every sequence of geodesics $\gamma_{(x_n,s_n)}^{(y_n,t_n)}$ is precompact in the uniform topology, and every subsequential limit is a geodesic from $(x,s)$ to $(y,t)$. 
\end{proposition}

\begin{proposition}[{\cite[Lemma B.12]{BSS22}, \cite[Lemma 3.3]{DSV22}, \cite[Theorem 2]{Bha23}}]
  \label{prop:19}
  The following holds with probability $1$. Let $u=(x,s;y,t)\in \RR^4_\uparrow$ and $u_n=(x_n,s_n;y_n,t_n)\in \RR^4_\uparrow$ be a sequence converging to $u$ admitting geodesics $\gamma_{(x_n,s_n)}^{(y_n,t_n)}$ which converge uniformly to a geodesic $\gamma_{(x,s)}^{(y,t)}$. Then the geodesics $\gamma_{(x_n,s_n)}^{(y_n,t_n)}$ converge to $\gamma_{(x,s)}^{(y,t)}$ in the overlap sense, by which we mean that the set
  \begin{equation}
    \label{eq:155}
    \left\{r\in [s_n,t_n]\cap [s,t]\colon \gamma_{(x_n,s_n)}^{(y_n,t_n)}(r)=\gamma_{(x,s)}^{(y,t)}(r)\right\}
  \end{equation}
  is an interval whose endpoints converge to $s$ and $t$.
\end{proposition}
We note that in \cite[Lemma B.12]{BSS22} and \cite[Lemma 3.3]{DSV22}, the above result is stated under some extra conditions related to geodesic uniqueness. However, by applying the result on the non-existence of geodesic bubbles established in \cite[Theorem 1]{Bha23} (see \cite[Lemma 3.3]{Dau23+} for an alternate proof), it can be seen that these extra conditions are not required. In fact, by using the result on the non-existence of geodesic bubbles mentioned above, one can also obtain the following approximation result for general geodesics, semi-infinite geodesics and interfaces.
\begin{proposition}
  \label{prop:23}
  The following hold with probability $1$.
  \begin{enumerate}
  \item For any fixed point $p=(x,s)$ and all points $q=(y,t)$ such that $(p;q)\in \RR_\uparrow^4$, any geodesic $\gamma_p^q$, and any point $t'$ satisfying $s<t'<t$, for all $\varepsilon>0$ small enough, there exists a neighbourhood $V_\varepsilon\ni (\gamma_p^q(t'),t')$ such that for all $q'\in V_\varepsilon$ and all geodesics $\gamma_p^{q'}$, we have $\gamma_{p}^{q'}\lvert_{[s,t'-\varepsilon]}=\gamma_p^q\lvert_{[s,t'-\varepsilon]}$.
    \item For all points $(p;q)=(x,s;y,t)\in \RR_\uparrow^4$, any geodesic $\gamma_p^q$, any points $s',t'$ satisfying $s<s'<t'<t$ and all $\varepsilon>0$ small enough, there exist neighbourhoods $U_\varepsilon \ni (\gamma_p^q(s'),s'),V_\varepsilon\ni (\gamma_p^q(t'),t')$, such that for all points $p'\in U_\varepsilon, q'\in V_\varepsilon$ and all geodesics $\gamma_{p'}^{q'}$, we have $\gamma_{p'}^{q'}\lvert_{[s'+\varepsilon,t'-\varepsilon]}=\gamma_p^q\lvert_{[s'+\varepsilon,t'-\varepsilon]}$.
  \item For all points $p=(y,t)$, any geodesic $\Gamma_p$, any $t'<t$, and any $\varepsilon>0$, there exists a neighbourhood $V_\varepsilon\ni (\Gamma_p(t'),t')$ with the property that for all $q\in V_\varepsilon$ and all geodesics $\Gamma_q$, we have $\Gamma_p\lvert_{(-\infty,t'-\varepsilon)}=\Gamma_q\lvert_{(-\infty,t'-\varepsilon)}$.
  \item For all points $p=(y,t)$, any interface $\Upsilon_p$, any $t'>t$, and any $\varepsilon>0$, there exists a neighbourhood $V_\varepsilon\ni (\Upsilon_p(t'),t')$ with the property that for all $q\in V_\varepsilon$ and all interfaces $\Upsilon_q$, we have $\Upsilon_p\lvert_{(t'+\varepsilon,\infty)}=\Upsilon_q\lvert_{(t'+\varepsilon,\infty)}$.
  \end{enumerate}
\end{proposition}
\begin{proof}
  Parts (3), (4) in the above follow immediately from \cite[Lemma 21, Theorem 5]{Bha23}. For (1), by \cite[Theorem 1]{Bha23}, the geodesic $\gamma_p^{q}\lvert_{[s,t']}$ must be the unique geodesic between its endpoints, and similarly, in (2), the geodesic $\gamma_p^q\lvert_{[s',t']}$ must be the unique geodesic between its endpoints. With the above uniqueness at hand, the result is straightforward to obtain by using Propositions \ref{prop:18}, \ref{prop:19}.
\end{proof}
Now, in \cite{BSS22}, it was shown that there exists a random countable set $\Xi_\downarrow\subseteq \RR$ such that for all points $\theta\in \Xi^c_\downarrow$, all geodesics $\Gamma_p^\theta$ for $p\in \RR^2$ eventually coalesce with each other. In the following lemma, we argue that whenever we have two almost disjoint right-most geodesics emanating from a point along different directions $\psi<\phi$, then there must exist an angle $\theta_*\in [\psi,\phi]$ with the property that there exist two almost disjoint $\theta_*$-directed downward geodesics emanating from $p$.
\begin{lemma}
  \label{lem:18}
Almost surely, for any point $p=(y,t)\in \RR^2$ and angles $\psi<\phi\in \RR$ admitting almost disjoint geodesics $\overline{\Gamma}_p^\psi$ and $\overline{\Gamma}_p^\phi$, there exists an angle $\theta_*\in [\psi,\phi]$ and a downward $\theta_*$-directed geodesic $\Gamma_p^{\theta_*}$ satisfying
  \begin{equation}
    \label{eq:168}
    \overline{\Gamma}_p^\psi(s)\leq \Gamma_p^{\theta_*}(s)<\overline{\Gamma}_p^{\theta_*}(s)\leq \overline{\Gamma}_p^\phi(s)
  \end{equation}
for all $s< t$.
\end{lemma}

\begin{proof}
  We define the angle $\theta_*$ by
  \begin{equation}
    \label{eq:79}
    \theta_*=
    \sup\left\{
      \theta: \theta\in [\psi,\phi], \textrm{ all geodesics } \Gamma_p^\theta \textrm{ are almost disjoint with }\overline{\Gamma}_p^\phi
    \right\},
  \end{equation}
  and we note that the set considered above is non-empty as at least $\psi$ belongs to it. Further, by the above definition, we know that for any $\theta\in (\theta_*,\phi]$, there exists a geodesic $\Gamma_p^\theta$ intersecting non-trivially with $\overline{\Gamma}_p^\phi$. As a result of the above, it can be seen that, for all $\theta\in (\theta_*,\phi]$, the right-most geodesic $\overline{\Gamma}_p^\theta$ must agree with $\overline{\Gamma}_p^\phi$ precisely at an interval $[s,t]$ for some $s<t$.

 We assume that $\theta_*\in (\psi,\phi)$. At the end of the proof, it will be clear that the boundary cases are simpler. We now argue by contradiction that for any $\theta,\theta'$ satisfying $\psi<\theta<\theta_*<\theta'<\phi$, the geodesics $\overline{\Gamma}_p^{\theta},\overline{\Gamma}_p^{\theta'}$ must be almost disjoint. Suppose that for some $r<t$, we have $\overline{\Gamma}_p^\theta(r)=\overline{\Gamma}_p^{\theta'}(r)\neq \overline{\Gamma}_p^{\phi}(r)$, where the latter is true since $\overline{\Gamma}_p^\theta$ and $\overline{\Gamma}_p^{\phi}$ are disjoint. Then we can consider the concatenation of the geodesics $\overline{\Gamma}_p^\theta\lvert_{(-\infty,r]}$ and $\overline{\Gamma}_p^{\theta'}\lvert_{[r,t]}$ and this is a $\theta$-directed semi-infinite geodesic which non-trivially intersects $\overline{\Gamma}^\phi_p$. However, since $\theta<\theta_*$, this contradicts the definition of $\theta_*$ from \eqref{eq:79}.

  Now, we choose a sequence $\theta_n$ increasing to $\theta_*$ and a sequence $\theta_n'$ decreasing to $\theta_*$. By geodesic ordering along with the fact that uniform limits of geodesics are geodesics (Proposition \ref{prop:18}), it is not difficult to see that the geodesics $\overline{\Gamma}_p^{\theta_n}$ converge locally uniformly to a geodesic $\Gamma_p^{\theta^*}$. Similarly, the geodesics $\overline{\Gamma}_p^{\theta_n'}$ converge locally uniformly to the right-most geodesic $\overline{\Gamma}_p^{\theta_*}$. Now, we know that $\overline{\Gamma}_p^{\theta_n}$ and $\overline{\Gamma}_p^{\theta_n'}$ are almost disjoint for all $n$ and thus due to the above local uniform convergence along with Proposition \ref{prop:19}, we obtain that $\Gamma_p^{\theta_*}$ and $\overline{\Gamma}_p^{\theta_*}$ must be disjoint as well. This completes the proof when $\theta_*\in (\psi,\phi)$.

  Now, if $\theta_*=\psi$, we can similarly show that for any $\theta'\in (\psi,\phi]$, the geodesics $\overline{\Gamma}_p^\psi$ and $\overline{\Gamma}_p^{\theta'}$ must be almost disjoint. By then considering a sequence $\theta_n'$ decreasing to $\psi$ and following the same argument, we have the result. Similarly, in the case $\theta_*=\phi$, we know that $\overline{\Gamma}_p^\theta$ and $\overline{\Gamma}_p^\phi$ are almost disjoint for all $\theta\in [\psi,\phi)$, and thus the same argument works in this case as well.
  
\end{proof}
We note that an argument of the above type has appeared earlier in the literature (see Proposition 3.5 in the arXiv version of \cite{BGH21}).

As discussed in the introduction, the study of atypical stars, or the points where geodesic coalescence fails has received significant attention in the past few years. One of these results, namely \cite{Bha22}, which we now state, will be important to this work, and in particular, will be crucially used in the computation of the dimension in Theorem \ref{thm:3}.
\begin{proposition}[{\cite[Proposition 28]{Bha22}}]
  \label{lem:13}
Fix $\nu>0$, $\ualpha<\oalpha<0$ and $\kappa\in (0,1/2)$. Then for all $\delta$ small enough and for any $\theta$ with $|\theta|\leq \log^{1/2-\kappa} \delta^{-1}$, we have for all $s\in [\ualpha,\oalpha]$,
  \begin{equation}
    \label{eq:58}
    \PP(\exists x\in [\Gamma_{\0}(s)-\delta,\Gamma_{\0}(s)+\delta] \textrm{ admitting almost disjoint geodesics }  \Gamma_{(x,s)}, \Gamma_{(x,s)}^\theta)\leq \delta^{1-\nu}.
  \end{equation}
\end{proposition}

\subsection{Interfaces in the directed landscape}
\label{sec:interf}
In exponential last passage percolation (LPP), an integrable model of random geometry on the planar lattice which is known \cite{RV21} to converge to the directed landscape, competition interfaces arise as regions of space in between two competing growth clusters \cite{FMP06}. In fact, as shown in \cite{Pim16}, such competition interfaces can also be interpreted as paths which are ``dual'' to semi-infinite geodesics, in the sense that they live on the dual lattice $(\ZZ^2)^*$ and do not cross the semi-infinite geodesics on the primal lattice $\ZZ^2$.

In the directed landscape, as established in \cite{Bha23}, the above duality takes the following form. Recall the $\theta$-directed semi-infinite geodesics as defined in the introduction. For any fixed $\theta\in \RR$, and in fact, for all $\theta\in \Xi_\downarrow^c$ for a random countable set $\Xi_\downarrow$, it is known that that all these geodesics coalesce \cite{BSS22}, in the sense that, almost surely, for any $p,q\in \RR^2$, any geodesics $\Gamma^\theta_p,\Gamma^\theta_q$ eventually merge. In fact, one can consider the ``geodesic tree'' $\cT^\theta_\downarrow$ defined by
\begin{equation}
  \label{eq:108}
  \cT_\downarrow^\theta=\bigcup_{p\in \RR^2}\inte(\Gamma_p^\theta),
\end{equation}
an object shown to in fact a.s.\ be a tree simultaneously for all $\theta\in \Xi_\downarrow^c$ in \cite{Bha23}. Now, for all $\theta\in \Xi_\downarrow^c$, one can consider the ``interface portrait'', which is the dual object defined by
\begin{equation}
  \label{eq:109}
  \cI_\uparrow^\theta= \bigcup_{\pi: \inte(\pi)\cap \cT_\downarrow^\theta=\emptyset}\inte(\pi),
\end{equation}
where the union is over all upward semi-infinite paths satisfying $\inte(\pi)\cap \cT_\downarrow^\theta=\emptyset$. In fact, if we have a point $p=(x,s)$ and a path $\pi \colon [s,\infty)\rightarrow \RR$ with $\pi(s)=x$ and $\inte(\pi)\cap \cT^\theta_\downarrow=\emptyset$, then we say that $\pi$ is an upward $\theta$-directed interface emanating from $p$ and denote it by $\Upsilon_p^\theta$. For any fixed $\theta\in \RR$ and fixed $p$, it can be shown that the path $\Upsilon_p^\theta$ is a.s. unique, but there do exist exceptional points $p$ where this is not true. In fact, as shown in \cite{Bha23} by using a duality present in discrete exponential LPP \cite{Pim16}, for any fixed $\theta\in \RR$, the interface portrait $\cI^\theta_\uparrow$ has the same law as the geodesic tree $\cT^\theta_\downarrow$ up to a reflection. In this work, we shall only need this duality for one semi-infinite geodesic, and we now state this. We note that we shall often work with $\theta=0$, and in this case, we omit $\theta$ from the notation. That is, the objects $\Gamma_p,\Upsilon_p,\cT_\downarrow,\cI_\uparrow$ shall refer to $0$-directed objects.

\begin{proposition}[{\cite[Proposition 2.8]{GZ22}, \cite[Lemma 4.20]{RV21}}]
  \label{prop:11}
 Almost surely, the curve $s\mapsto \Gamma_{\0}(-s)$ has the same law as the interface $\Upsilon_{\0}$. 
\end{proposition}
We note that as a result of the above and Proposition \ref{prop:17}, the interface $\Upsilon_{\0}$ is a.s.\ locally $2/3-$ H\"older continuous. Throughout the paper, we shall also require the following rational approximation result from \cite{Bha23} which is an immediate consequence of Proposition \ref{prop:23}.
\begin{proposition}
  \label{prop:15}
  Almost surely, we have $\cT_\downarrow=\bigcup_{p\in \mathbb{Q}^2}\inte(\Gamma_p)$ and  $\cI_\uparrow =\bigcup_{p\in \mathbb{Q}^2}\inte (\Upsilon_p)$. Similarly, the geodesic frame $\cW$ from the statement of Theorem \ref{thm:3} is a.s.\ equal to the union of interiors of geodesics between rational points.
\end{proposition}
Apart from the interfaces $\Upsilon_{p}$ which are dual to downward $0$-directed semi-infinite geodesics, we shall also need a more general notion of an interface emanating from a given initial condition, where by an initial condition, we mean a function $\fh\colon \RR\rightarrow \RR\cup \{-\infty\}$ which is not identically equal to $-\infty$, is upper semi-continuous and satisfies $\fh(x)\leq c(1+|x|)$ for some constant $c>0$. Such a notion of an interface has been defined in \cite{RV21}, and we now discuss this definition. First, we shall need the notion of geodesics to an initial condition which we now define. For any point $q=(y,t)$ with $t>0$, a geodesic from $(\fh,0)$ to $q$ can be defined as any path $\gamma_{(\fh,0)}^{q}\colon[0,t]\rightarrow \RR$ with $\gamma_{(\fh,0)}^{q}(t)=y$ satisfying
\begin{equation}
  \label{eq:30}
  \ell(\gamma_{(\fh,0)}^{q};\cL)+\fh(\gamma_{(\fh,0)}^{q}(0))=\sup_\xi\{\ell(\xi;\cL)+\fh(\xi(0))\},
\end{equation}
where the supremum is over all paths $\xi\colon[0,t]\rightarrow \RR$ satisfying $\xi(t)=y$. It can be shown that such geodesics always exist \cite[Lemma 3.2]{RV21} and in-fact, we always have well-defined left-most and right-most geodesics $\underline{\gamma}_{(\fh,0)}^q, \overline{\gamma}_{(\fh,0)}^q$. Further, it can be shown \cite[Lemma 3.7]{RV21} that for any fixed initial condition $\fh$ and any fixed point $q$, there a.s.\ exists a unique geodesic $\gamma_{(\fh,0)}^{q}$.

Now, for any $x\in \RR$, one can define \cite[Section 4.1]{RV21} the competition function $d_x\colon \RR\times (0,\infty)\rightarrow \RR$ by
\begin{equation}
  \label{eq:110}
  d_x(y,t)=\sup_{x'\geq x}\{\fh(x')+\cL(x',0;y,t)\}-\sup_{x'\leq x}\{\fh(x')+\cL(x',0;y,t)\},
\end{equation}
and it can be checked \cite[Proposition 4.1]{RV21} that the above is monotonically increasing in $y$ for every fixed $t$. Using the above, one can define \cite[Definition 4.2]{RV21} the left and right interfaces $\underline{\Upsilon}_{(x,0)}^\fh,\overline{\Upsilon}_{(x,0)}^\fh$ by
\begin{align}
  \label{eq:112}
  \underline{\Upsilon}_{(x,0)}^\fh(t)&=\inf\{y\in \RR\colon d_x(y,t)\geq 0\},\nonumber\\
  \overline{\Upsilon}_{(x,0)}^\fh(t)&=\sup\{y\in \RR\colon d_x(y,t)\leq 0\}
\end{align}
for $t>0$.

Now, throughout this section, we shall only consider points $x$ for which there exist points $\underline{x},\overline{x}$ satisfying $\underline{x}<x<\overline{x}$ and $\fh(\underline{x}),\fh(\overline{x})>-\infty$, and for this section, we call such points $x$ as \emph{interior} points. It can be shown that almost surely, for any fixed initial condition $\fh$ and for all interior points $x$, the objects $\underline{\Upsilon}_{(x,0)}^\fh,\overline{\Upsilon}_{(x,0)}^\fh$ are genuine continuous paths \cite[Proposition 4.4, Proposition 4.5]{RV21}. If it turns out that $\underline{\Upsilon}^\fh_{(x,0)}=\overline{\Upsilon}^\fh_{(x,0)}$ for some $x$, then we simply use $\Upsilon^\fh_{(x,0)}$ to denote this unique path, and in this paper, we shall mostly be concerned with the setting where this uniqueness holds. One important property is that interfaces corresponding to an initial condition $\fh$ tend to avoid geodesics to it, and we now state a result in this direction.
\begin{lemma}[{\cite[Lemma 5.7]{RV21}}]
  \label{lem:32}
  For any fixed initial condition $\fh$, the following a.s.\ hold for all interior $x\in \RR$ admitting a unique interface $\Upsilon_{(x,0)}^\fh$.
  \begin{enumerate}
  \item For any $q\notin \Upsilon_{(x,0)}^\fh$, all geodesics $\gamma_{(\fh,0)}^q$ never intersect $\Upsilon_{(x,0)}^\fh$
  \item For all points $q=(x,s) \in \inte(\Upsilon_{(x,0)}^\fh)$ and for any geodesic $\gamma_{(\fh,0)}^q$, there exists $r\in (0,s]$ such that we have $\gamma_{(\fh,0)}^q(s')=\Upsilon_{(x,0)}^\fh(s')$ for all $s'\in [r,s]$ and we have $\gamma_{(\fh,0)}^q\lvert_{(0,s)}\subseteq \cS_{\Upsilon_{(x,0)}^\fh}^*$ for exactly one $*\in \{\tL,\tR\}$.
  \end{enumerate}
\end{lemma}

We now state a lemma arguing that for points on an interface, there do exist left-most and right-most geodesics which are almost disjoint with the interface.

\begin{lemma}
  \label{lem:19}
Fix an initial condition $\fh$. Almost surely, for all interior $x\in \RR$ admitting a unique interface $\Upsilon_{(x,0)}^\fh$, and for all $s>0$ and $s'\in (0,s)$, we have
  \begin{equation}
    \label{eq:161}
    \underline{\gamma}_{(\fh,0)}^{(\Upsilon_{(x,0)}^\fh(s),s)}(s')< \Upsilon_{(x,0)}^\fh(s')<\overline{\gamma}_{(\fh,0)}^{(\Upsilon_{(x,0)}^\fh(s),s)}(s'),
  \end{equation}
 and further, we also have
  \begin{equation}
    \label{eq:162}
    \underline{\gamma}_{(\fh,0)}^{(\Upsilon_{(x,0)}^\fh(s),s)}(0)< x< \overline{\gamma}_{(\fh,0)}^{(\Upsilon_{(x,0)}^\fh(s),s)}(0).
  \end{equation}
\end{lemma}
\begin{proof}
We first prove \eqref{eq:161}. For any $s>0$, let $q=(\Upsilon_{(x,0)}^{\fh}(s),s)$ and define the sequences $q_n^{\tL}=(\Upsilon_{(x,0)}^{\fh}(s)-1/n,s)$ and $q_n^{\tR}=(\Upsilon_{(x,0)}^{\fh}(s)+1/n,s)$. By a precompactness argument employing Proposition \ref{prop:18}, it can be seen that the geodesics $\underline{\gamma}_{(\fh,0)}^{q_n^{\tL}}$,  $\overline{\gamma}_{(\fh,0)}^{q_n^{\tL}}$ must converge subsequentially in the uniform topology to the geodesics $\underline{\gamma}_{(\fh,0)}^{q}, \overline{\gamma}_{(\fh,0)}^{q}$ respectively. Further, by Proposition \ref{prop:19}, the above convergence must also occur in the overlap sense.

Now, by invoking (1) in Lemma \ref{lem:32}, we know that for all $n$ and $s'\in (0,s]$,
 \begin{equation}
   \label{eq:156}
      \underline{\gamma}_{(\fh,0)}^{q_n^{\tL}}(s')<\Upsilon_{(x,0)}^{\fh}(s')< \overline{\gamma}_{(\fh,0)}^{q_n^{\tR}}(s').
    \end{equation}
    On taking the limit $n\rightarrow \infty$ along the above-mentioned subsequence and using the overlap sense above, we obtain that we must have
  \begin{equation}
    \label{eq:137}
    \underline{\gamma}_{(\fh,0)}^{q}(s')<\Upsilon_{(x,0)}^{\fh}(s')< \overline{\gamma}_{(\fh,0)}^{q}(s')
  \end{equation}
  for all $s'\in (0,t)$. Indeed, simply taking a limit of \eqref{eq:156} yields \eqref{eq:137} with the $<$ signs replaced by $\leq $ signs, and then one can note that these have to be strict inequalities as otherwise, due to the overlap sense convergence, the inequalities in \eqref{eq:156} would not be strict.

  We now come to \eqref{eq:161}, and we just prove the first inequality, since the proof of the latter is analogous. First, we note that if we had $\underline{\gamma}_{(\fh,0)}^{(\Upsilon_{(x,0)}^\fh(s),s)}(0)>x$, then we must have
  \begin{equation}
    \label{eq:163}
    \sup_{x'\leq x}\{\fh(x')+\cL(x',0;\Upsilon_{(x,0)}^\fh(s),s)\}< \sup_{x'\geq x}\{\fh(x')+\cL(x',0;\Upsilon_{(x,0)}^\fh(s),s)\},
  \end{equation}
  thereby contradicting the very definition of the interface $\Upsilon_{(x,0)}^\fh$ from \eqref{eq:112}. Finally, if we had $\underline{\gamma}_{(\fh,0)}^{(\Upsilon_{(x,0)}^\fh(s),s)}(0)=x$, then by using Propositions \ref{prop:18}, \ref{prop:19}, it can be seen that for all large enough $n\in \NN$, we must also have $\underline{\gamma}_{(\fh,0)}^{(\Upsilon_{(x,0)}^\fh(s)-n^{-1},s)}(0)=x$. This would imply that $d_x(\cdot,s)$ is zero on an interval and would contradict the assumed uniqueness of the interface $\Upsilon_{(x,0)}^\fh$.
\end{proof}

Also, the work \cite{RV21} provides a basic condition which implies the uniqueness of the above-mentioned interface.

\begin{proposition}[{\cite[Proposition 4.9]{RV21}}]
  \label{prop:16}
 Let $B$ be a two sided Brownian motion with diffusivity $2$ independent of $\cL$. For a fixed initial condition $\fh$, a point $x$ is said to be a polar point for $\fh$ if for some interval $[a,b]$ satisfying $x\in [a,b]$ and $\sup_{y\in[a,b]}\fh(y)>-\infty$, we have
  \begin{equation}
    \label{eq:113}
    \PP(\fh(x)+B(x)=\sup_{y\in [a,b]}\{\fh(y) + B(y)\})>0.
  \end{equation}
 Then almost surely, we have a unique interface $\Upsilon_{(x,0)}^\fh$ for all $x$ which are not polar for $\fh$.
\end{proposition}
Since points $x$ where $\fh(x)=-\infty$ are always non-polar, one has the following immediate consequence.
\begin{lemma}
  \label{lem:24}
  Let $\fh$ be an initial condition which is not identically equal to $-\infty$. Then one has a unique interface $\Upsilon_{(x,0)}^\fh$ for all interior points $x$ where $\fh(x)=-\infty$.
\end{lemma}

\subsection{Busemann functions }
\label{sec:busemann}
Apart from using distances $\cL(p;q)$ for points $(p;q)\in \RR_\uparrow^4$, we shall also need to use ``distances to infinity'', and these can be formalized by Busemann functions, a notion originally studied in geometry \cite{Bus12} and subsequently introduced to first passage percolation in \cite{New95,Hof05}. We recall that for any fixed $\theta\in \RR$, almost surely, any two geodesics $\Gamma^\theta_p,\Gamma^\theta_q$ eventually coalesce \cite{RV21,GZ22,BSS22}. As a result, the Busemann function $\cB_\theta$ can be defined as
\begin{equation}
  \label{eq:100}
  \cB_\theta(p;q)=\cL(z;p)-\cL(z;q),
\end{equation}
where $z$ is any point satisfying $z\in \Gamma^\theta_p\cap \Gamma^\theta_q$ for some geodesics $\Gamma^\theta_p,\Gamma^\theta_q$. It is not difficult to show that the above notion is well-defined regardless of the precise choice of $\Gamma^\theta_p,\Gamma^\theta_q$ and $z$. Further, $\cB_\theta$ can be shown to be a.s.\ continuous \cite[Theorem 5.1]{BSS22} in both the parameters $p$ and $q$. Often, we shall work with the Busemann function restricted to a points having a fixed temporal coordinate $s$, and thus, we define $\cB_\theta^s\colon \RR\rightarrow \RR$ by
\begin{equation}
  \label{eq:102}
  \cB_\theta^s(x)=\cB_\theta( (x,s), (0,s)).
\end{equation}
The following precise description of the law of the above shall be useful to us.
\begin{proposition}[{see e.g.\ \cite[Lemma 4.4]{GZ22}}]
  \label{prop:22}
  For any fixed $s,\theta\in \RR$, the process $\cB_\theta^s$ is a two-sided Brownian motion with drift $2\theta$.
\end{proposition}
In fact, there is a precise description (\cite{Bus21}, \cite[Theorem 3.3, Appendix D]{BSS22}) of the joint law of $(\cB^s_{\theta_1},\dots,\cB^s_{\theta_n})$ for any fixed $s$ and finitely many angles $\theta_1<\dots<\theta_n$. However, for the present paper, we shall only require the above simpler Proposition \ref{prop:22}.

We now record an easy lemma relating the interface $\Upsilon^\theta_{\0}$ defined earlier as a path starting from $0$ and dual to the tree $\cT^\theta_\downarrow$ to the interface defined using the initial condition given by the Busemann function.
\begin{lemma}[{see e.g.\ \cite[Lemma 31]{Bha23}}]
  \label{lem:33}
 For any fixed $\theta\in \RR$, we a.s.\ have a unique interface $\Upsilon_{\0}^{\cB_\theta^0}$ and further, we have the equality $\Upsilon^\theta_{\0}=\Upsilon_{\0}^{\cB_\theta^0}$.
\end{lemma}

\subsection{Disjoint optimizers and the extended directed landscape}
\label{sec:disj-optim}
We now discuss results from the work \cite{DZ21} on disjoint optimizers in the directed landscape, and later, Theorem \ref{thm:2} will be obtained as an application of these results. To begin, we provide the definition of the extended directed landscape.

\begin{definition}[{\cite[Definition 1.1]{DZ21}}]
  For $k\in \NN$, define the set $\RR_{\leq}^k=\{\mathbf{x}\in \RR^k: x_1\leq \dots \leq x_k\}$. Now, for $\mathbf{x},\mathbf{y}\in \RR_{\leq}^k$ for some $k\in \NN$, we define
  \begin{equation}
    \label{eq:157}
    \cL^*(\mathbf{x},s;\mathbf{y},t)=\sup_{\xi_1,\dots ,\xi_k}\sum_{i=1}^k\ell(\xi_i;\cL),
  \end{equation}
  where the above supremum is over pairwise almost disjoint paths $\xi_i$ from $(x_i,s)$ to $(y_i,t)$.
\end{definition}
Just as $\cL$ satisfies the reverse triangle inequality, we also have the following result.
\begin{proposition}[{\cite[Proposition 6.9]{DZ21}}]
  \label{prop:12}
  Fix $k\in \NN$. Almost surely, for all $\mathbf{x},\mathbf{z},\mathbf{y}\in \RR_{\leq}^k$ and all $s<r<t$, we have
  \begin{equation}
    \label{eq:158}
    \cL^*(\mathbf{x},s;\mathbf{y},t)\geq \cL^*(\mathbf{x},s;\mathbf{z},t)+ \cL^*(\mathbf{z},r;\mathbf{y},t).
  \end{equation}
\end{proposition}
As is the case for $\cL$, the extended landscape is also a.s.\ continuous.
\begin{proposition}[{\cite[Lemma 5.6,Lemma 6.5]{DV21}}]
  \label{prop:24}
  Almost surely, for any fixed $k\in \NN$, the function $(\mathbf{x},s;\mathbf{y},t)\mapsto \cL^*(\mathbf{x},s;\mathbf{y},t)$ is continuous for all $s<t$ and $\mathbf{x},\mathbf{y}\in \RR_{\leq}^k$.
\end{proposition}
We now state the main result of \cite{DZ21}.
\begin{proposition}[{\cite[Proposition 8.1]{DZ21}}]
  \label{prop:20}
  Almost surely, for all $k\in \NN$, all points $\mathbf{x},\mathbf{y}\in \RR_{\leq}^k$ and $s<t$, there exists a disjoint optimizer $\{\xi_i\}_{i=1}^k$ from $(\mathbf{x},s)$ to $(\mathbf{y},t)$, in the sense that the supremum in \eqref{eq:157} is attained by a collection of almost disjoint paths $\{\xi_i\}_{i=1}^k$.
\end{proposition}

\section{Estimates for the frequency of intersections of geodesics and interfaces}
\label{sec:proof-thm-dim}
In this section, we shall prove the following estimate on the frequency of intersections of the interface $\Upsilon_{\mathbf{0}}$ with geodesics $\gamma_p^q$. Recall that we always set $\varepsilon_n=2^{-n}$.
\begin{proposition}
  \label{prop:10}
  Almost surely, for all $p=(y',t'),q=(y,t)\notin \Upsilon_{\mathbf{0}}$ with $0<t'< t$, and any geodesic $\gamma_p^q$, we have
\begin{equation}
    \label{eq:74}
    \lim_{n\rightarrow \infty}\frac{\log\left( \#\{i\in [\![t'\varepsilon_n^{-1},t\varepsilon_n^{-1}]\!]:\gamma_p^q(s)=\Upsilon_{\mathbf{0}}(s) \textrm{ for some } s\in [i\varepsilon_n,(i+1)\varepsilon_n]\}\right)}{\log \varepsilon_n^{-1}}=0.
  \end{equation}
\end{proposition}
We now use transversal fluctuation estimates to reduce Proposition \ref{prop:10} to a statement about the closeness of $\gamma_p^q$ and $\Upsilon_{\mathbf{0}}$ at a deterministic mesh of times, and then spend the remainder of this section proving this statement.
\begin{proposition}
  \label{lem:9}
  Almost surely, for any fixed $\beta>0$ and for all $p=(y',t'),q=(y,t)\notin \Upsilon_{\mathbf{0}}$ with $0<t'< t$, and any geodesic $\gamma_p^q$, we have
  \begin{equation}
    \label{eq:52}
    \limsup_{n\rightarrow \infty}\frac{\log\left(\#\{i\in [\![t'\varepsilon_n^{-1},t\varepsilon_n^{-1}]\!]:|\gamma_p^q(i\varepsilon_n)-\Upsilon_{\mathbf{0}}(i\varepsilon_n)|\leq \varepsilon_n^{2/3-\beta}\}\right)}{\log \varepsilon_n^{-1}}\leq 3\beta/2.
  \end{equation}
\end{proposition}
  \begin{proof}[Proof of Proposition \ref{prop:10} assuming Proposition \ref{lem:9}]
    First, by Proposition \ref{prop:17}, we know that on an almost sure set, the geodesics $\gamma_p^q$ are all H\"older $2/3-$ regular for all $(p,q)\in \RR_\uparrow^4$. As a consequence, almost surely, for any fixed $\beta>0$ and any $p=(y',t'),q=(y,t)\notin \Upsilon_{\0}$ with $0<t'<t$, and all $n$ large enough, we have
    \begin{equation}
      \label{eq:165}
      \sup_{s,s'\in [t',t], |s-s'|\leq \varepsilon_n}|\gamma_p^q(s)- \gamma_p^q(s')|\leq  \varepsilon_n^{2/3-\beta}/2,
    \end{equation}
  and we now fix a $\beta>0$ and points $p,q$ as above as above for the rest of the proof. By Propositions \ref{prop:17}, \ref{prop:11}, we know that almost surely, the interface $\Upsilon_{\mathbf{0}}$ is locally $2/3-$ H\"older continuous as well, and thus we have
    \begin{equation}
      \label{eq:75}
      \sup_{s,s'\in [t',t], |s-s'|\leq \varepsilon_n}|\Upsilon_{\mathbf{0}}(s)- \Upsilon_{\mathbf{0}}(s')|\leq  \varepsilon_n^{2/3-\beta}/2
    \end{equation}
    for all $n$ large enough. As an immediate consequence, for all $n$ large enough, we have the inclusions
       \begin{align}
      \label{eq:101}
     & \{i\in  [\![t'\varepsilon_n^{-1},t\varepsilon_n^{-1}]\!]: \gamma_p^q(s)=\Upsilon_{\0}(s) \textrm{ for some } s\in [i\varepsilon_n,(i+1)\varepsilon_n]\}\nonumber\\
      &\subseteq  \{i\in  [\![t'\varepsilon_n^{-1},t\varepsilon_n^{-1}]\!]: \gamma_p^q(s),\Upsilon_{\0}(s)\in [\gamma_p^q(i\varepsilon_n)-\varepsilon_n^{2/3-\beta},\gamma_p^q(i\varepsilon_n)+\varepsilon_n^{2/3-\beta}] \textrm{ for all } s\in [i\varepsilon_n,(i+1)\varepsilon_n]\cap[t',t]\}\nonumber\\
&\subseteq  \{i\in  [\![t'\varepsilon_n^{-1},t\varepsilon_n^{-1}]\!]: |\gamma_p^q(i\varepsilon_n)-\Upsilon_{\mathbf{0}}(i\varepsilon_n)|\leq \varepsilon_n^{2/3-\beta}\}.
       \end{align}
    By using this along with Proposition \eqref{lem:9}, we obtain that the limsup of the quantity in \eqref{eq:74} is at most $3\beta/2$. Since $\beta$ was arbitrary, we obtain that the limsup must in fact a.s.\ be $0$ and this completes the proof.
  \end{proof}

  In fact, by using the H\"older continuity argument presented above, we can use Proposition \ref{lem:9} to complete the proof of Theorem \ref{thm:3} as follows.
  \begin{proof}[Proof of Theorem \ref{thm:3} assuming Proposition \ref{lem:9}]
First, we show that almost surely, for all $(p;q)=(x,s;y,t)\in \RR_\uparrow^4$ and any geodesic $\gamma_p^q$, the set of times $r$ for which $\gamma_p^q(r)=\Upsilon_{\0}(r)$ occurs, a.s.\ has Hausdorff dimension zero. It is easy to see that it is sufficient to assume $s\geq 0$ in the above, and we do assume so. Now, we note that as a consequence of (2) in Proposition \ref{prop:23}, on an almost sure set, for any $\varepsilon>0$, we can find points $\widetilde{p},\widetilde{q}=(\widetilde{x},\widetilde{s}),(\widetilde{y},\widetilde{t})\notin \Upsilon_{\0}$ with $0<\widetilde{s}<\widetilde{t}$ such that for any geodesic $\gamma_{\widetilde{p}}^{\widetilde{q}}$, we have $\gamma_{\widetilde{p}}^{\widetilde{q}}\lvert_{[s+\varepsilon,t-\varepsilon]}=\gamma_{p}^{q}\lvert_{[s+\varepsilon,t-\varepsilon]}$. As a result, by Proposition \ref{prop:10}, we immediately obtain that the Hausdorff dimension of the set of times $r\in [s+\varepsilon,t-\varepsilon]$ for which $\gamma_{p}^{q}(r)=\Upsilon_{\0}(r)$ is almost surely zero. Now, by sending $\varepsilon$ to zero and using the countable stability of Hausdorff dimension, we obtain the desired result.

   The next task is to show that $\cI_\uparrow \cap \cW$ is a.s.\ an infinite set satisfying $\dim_{\mathrm{KPZ}}(\cI_\uparrow\cap \cW)=0$, where the latter refers to the Hausdorff dimension with respect to $d_{\mathrm{KPZ}}$. Just by a topological argument where we choose $p_n,q_n$ to be on different sides of $\Upsilon_{\0}$, it is not difficult to see that there must almost surely exist infinitely many points $(p_n;q_n)\in \RR_\uparrow^4$, such that any geodesics $\gamma_{p_n}^{q_n}$ intersect $\Upsilon_{\0}$ at different times, and as a result, we obtain that $\cI_\uparrow\cap \cW$ must almost surely be an infinite set.

It remains to establish that $\dim_{\mathrm{KPZ}}(\cI_\uparrow \cap \cW)=0$ almost surely. By the countable stability of Hausdorff dimension along with the rational approximation result Proposition \ref{prop:15}, it suffices to fix rational points $p=(x_0,s_0),q=(y_0,t_0)$ with $0<s_0<t_0$ and establish that $\dim_{\mathrm{KPZ}}(\gamma_p^q\cap \Upsilon_{\0})=0$ almost surely. Now, consider the set defined by
    \begin{equation}
      \label{eq:103}
      \{s\in [s_0,t_0], \gamma_p^q(s)=\Upsilon_{\0}(s)\},
    \end{equation}
    and let $N_n$ denote the minimum number of boxes of size $ \varepsilon_n^{2/3}\times \varepsilon_n$ required to cover the set \eqref{eq:103}. By using the definition of $d_{\mathrm{KPZ}}$, it is easy to see that we need only establish the a.s.\ convergence
    \begin{equation}
      \label{eq:105}
      \lim_{n\rightarrow \infty}\frac{\log N_n}{\log \varepsilon_n^{-1}} =0.
    \end{equation}
    In fact, instead of working with $N_n$ directly, we first fix $\beta>0$ and define $N^{\beta}_n$ as the minimum number of boxes of size $ (2\varepsilon^{2/3-\beta})\times \varepsilon_n$ required to cover the set in \eqref{eq:103}. Now, by the inclusions in \eqref{eq:101} along with Proposition \ref{lem:9} and the fact that $p,q\notin \Upsilon_{\0}$ almost surely, which holds since $p,q$ are fixed points, we know that almost surely,
    \begin{equation}
      \label{eq:106}
      \limsup_{n\rightarrow \infty} \frac{\log N_n^\beta}{\log \varepsilon_n^{-1}}\leq 3\beta/2
    \end{equation}
 
 The task now is to use \eqref{eq:106} to obtain \eqref{eq:105}, but this is easy. Indeed, since any $(2\varepsilon_n^{2/3-\beta})\times \varepsilon_n$ box can be deterministically covered by $2\varepsilon_n^{-\beta}$ many $ \varepsilon_n^{2/3}\times \varepsilon_n$ boxes, we immediately have $N_n\leq 2\varepsilon_n^{-\beta} N_n^\beta$ for all $n$. As a result, a.s.\ for any fixed $\beta>0$, we obtain
    \begin{equation}
      \label{eq:104}
      \limsup_{n\rightarrow \infty}\frac{\log N_n}{\log \varepsilon_n^{-1}}\leq 3\beta/2 + \beta=5\beta/2.
    \end{equation}
    Now, since $\beta>0$ is arbitrary, the above limit must in fact a.s.\ be equal to zero, and this yields \eqref{eq:105}, thereby completing the proof.
  \end{proof}
In the remainder of this section, we seek to prove Proposition \ref{lem:9}. As mentioned in Section \ref{sec:outline}, instead of attempting to prove it directly, we shall take a more indirect route. Indeed, rather than working with the interface $\Upsilon_{\mathbf{0}}$ directly, we shall first establish corresponding estimates for the paths $\Upsilon_{\mathbf{0}}^{X,K}\coloneqq \Upsilon_{\0}^{\fh_X^K}$ defined as interfaces (see Section \ref{sec:interf}) to certain compactly supported initial conditions $\fh_X^K$ with $K>0$ being a parameter that will be sent to $\infty$ at the end. Also, instead of working with finite geodesics $\gamma_p^q$, we shall first argue for semi-infinite geodesics for the landscape $\wcL$ formed by independently resampling the portion of $\cL$ below the time line $0$. We shall use a $\sim$ in the superscript to signify objects  for the landscape $\wcL$; for instance, $\wGamma_q$ denotes a downward $0$-directed semi-infinite $\wcL$-geodesic emanating from $q$. We now state the central estimate that will be the focus of this section, and we provide its proof in the next subsection. Subsequently, we shall use it to prove Proposition \ref{lem:9}.

\begin{proposition}
  \label{lem:1}
  Fix $K,M,\nu>0$, $\oalpha>\ualpha>0$, and a point $q=(y_0,t_0)$ with $t_0>\oalpha$. Let $X\sim \mathrm{Unif}([-M,M])$ be a random variable independent of $\sigma(\cL,\wcL)$. Define the initial condition $\fh_X^K(x)=\cB_0^0(x)+2Xx$ for all $|x|\in [K^{-1},K]$ and $-\infty$ otherwise and use $\Upsilon^{X,K}_{\mathbf{0}}$ to denote the a.s.\ unique (Lemma \ref{lem:24}) upward $\cL$-interface emanating from $\0$ for the initial condition given by $\fh_X^K$. Then, uniformly for all $s\in [\ualpha,\oalpha]$ and all $\delta$ small enough, we have
  \begin{equation}
    \label{eq:9}
    \PP( |\wGamma_q(s)-\Upsilon^{X,K}_{\mathbf{0}}(s)|\leq \delta)\leq \delta^{3/2-\nu}.
  \end{equation}
\end{proposition}
For notational convenience, for an initial condition $\fh$ and a set $A\subseteq \RR$, we define $\fh\lvert_A\colon \RR\rightarrow \RR\cup \{-\infty\}$ by $\fh\lvert_A(x)=\fh(x)$ for $x\in A$ and $\fh\lvert_A(x)=\infty$ for $x\notin A$. In the remainder of this section, we shall frequently work with the initial conditions $\fh_X^K\lvert_{(-\infty,0]}$ and $\fh_X^K\lvert_{[0,\infty)}$. %
\subsection{The proof of Proposition \ref{lem:1}}
\label{sec:proof-prop}

The goal now is to provide the proof of Proposition \ref{lem:1}. We shall frequently work in the setting of the above-mentioned proposition, that is, we shall often have the positive constants $\ualpha,\oalpha,K,M,\nu$ as in Proposition \ref{lem:1} satisfying $\oalpha>\ualpha>0$. The following lemma is not difficult but important for us. 

\begin{lemma}
  \label{lem:17}
  Fix $\ualpha,\oalpha,K,M>0$ and $q=(y_0,t_0)$ with $t_0>\oalpha$. For any $s\in [\ualpha,\oalpha]$ and for all $\delta>0$, we have
  \begin{equation}
    \label{eq:77}
    \{|\wGamma_q(s)-\Upsilon_{\mathbf{0}}^{X,K}(s)|\leq \delta\} \subseteq \{\exists x\in [\wGamma_q(s)-\delta,\wGamma_q(s)+\delta] :  \overline{\gamma}_{(-K,0)}^{(x,s)}, \overline{\gamma}_{(K,0)}^{(x,s)} \textrm{ are almost disjoint}\}.         
  \end{equation}
\end{lemma}

\begin{proof}
It suffices to establish that the geodesics $\overline{\gamma}_{(-K,0)}^{(\Upsilon_{\0}^{X,K}(s),s)}$ and $\overline{\gamma}_{(K,0)}^{(\Upsilon_{\0}^{X,K}(s),s)}$ are almost disjoint, and this is what we shall prove.  We now define $\gamma_1\coloneqq \overline{\gamma}_{(\fh_X^K\lvert_{(-\infty,0]},0)}^{(\Upsilon_{\0}^{X,K}(s),s)}$ and  $\gamma_2\coloneqq \overline{\gamma}_{(\fh_X^K\lvert_{[0,\infty)},0)}^{(\Upsilon_{\0}^{X,K}(s),s)}$. Due to the definition of the interface $\Upsilon_{\0}^{X,K}$, it is easy to see that both $\gamma_1,\gamma_2$ are also geodesics to the initial condition $\fh_X^K$. Now, we also know that $\fh_X^K(x)=-\infty$ for all $|x|\notin [K^{-1},K]$, and thus, we must have,
  \begin{equation}
    \label{eq:115}
    \gamma_1(0)\in [-K,-K^{-1}], \gamma_2(0)\in [K^{-1},K].
  \end{equation}
  By the first inclusion above along with (2) in Lemma \ref{lem:32}, we know that
  \begin{equation}
    \label{eq:166}
    \gamma_1(s')\leq \Upsilon_{\0}^{X,K}(s')
  \end{equation}
  for all $s'\in (0,s)$. Further, by the right-most nature of $\gamma_2$ along with \eqref{eq:161} from Lemma \ref{lem:19}, we know that
  \begin{equation}
    \label{eq:167}
    \Upsilon_{\0}^{X,K}(s')< \gamma_2(s')
  \end{equation}
  for all $s'\in (0,s)$.
  By combining \eqref{eq:115}, \eqref{eq:166} and \eqref{eq:167}, we obtain that $\gamma_1,\gamma_2$ are almost disjoint. Finally, by using \eqref{eq:115} along with geodesic ordering, we note that $\overline{\gamma}_{(-K,0)}^{(\Upsilon_{\0}^{X,K}(s),s)}$ must be to the left of $\gamma_1$ and $\overline{\gamma}_{(K,0)}^{(\Upsilon_{\0}^{X,K}(s),s)}$ must be to the right of $\gamma_2$. Thus, since $\gamma_1,\gamma_2$ are almost disjoint, we obtain that $\overline{\gamma}_{(-K,0)}^{(\Upsilon_{\0}^{X,K}(s),s)},\overline{\gamma}_{(K,0)}^{(\Upsilon_{\0}^{X,K}(s),s)}$ are almost disjoint as well, and this completes the proof.
\end{proof}

We now give a $\delta^{1-o(1)}$ estimate on the probability of the event appearing in the right hand side of the above lemma.
\begin{lemma}
  \label{lem:14}
   Fix $\ualpha,\oalpha,K,M,\nu>0$ and $q=(y_0,t_0)$ with $t_0>\oalpha$. Then uniformly over all $s\in [\ualpha,\oalpha]$, and for all small enough $\delta>0$, we have
  \begin{equation}
    \label{eq:59}
    \PP(\exists x\in [\wGamma_q(s)-\delta,\wGamma_q(s)+\delta] :  \overline{\gamma}_{(-K,0)}^{(x,s)}, \overline{\gamma}_{(K,0)}^{(x,s)}) \textrm{ are almost disjoint})\leq \delta^{1-\nu}. 
  \end{equation}
\end{lemma}
\begin{proof}
  For this proof, we set $\theta_\delta=\log^{9/20}\delta^{-1}$ and $m_\delta=\log^{2/5}\delta^{-1}$. Now, we begin by defining a few typical events which we shall truncate on.
  \begin{align}
    \label{eq:62}
    \cA_\delta&=\{ \sup_{t\in [\ualpha, \oalpha]}|\wGamma_q(t)|\leq m_\delta\},\nonumber\\
    \cC_\delta&=\{\wGamma^{-\theta_\delta}_{(m_\delta+\delta,s)}(0)\leq -K\}\cap \{ \wGamma^{\theta_\delta}_{(-m_\delta-\delta,s)}(0)\geq K\} \nonumber\\
    &= \{\wGamma^{-\theta_\delta}_{(x,s)}(0)\leq -K \textrm{ and } \wGamma^{\theta_\delta}_{(x,s)}(0)\geq K\textrm { for all } |x|\leq m_\delta+\delta \textrm{ and all geodesics } \wGamma^{-\theta_\delta}_{(x,s)}, \wGamma^{\theta_\delta}_{(x,s)} \}.
  \end{align}
  By Proposition \ref{lem:12}, we know that, uniformly over the choice of $s\in [\ualpha,\oalpha]$, the probability $1-\PP(\cA_\delta\cap \cC_\delta)$ decays superpolynomially to $0$ as $\delta\rightarrow 0$.

 Now, for $\theta\in \RR$, define the event $\cE_\delta^\theta$ by
  \begin{equation}
    \label{eq:78}
  \cE_\delta^\theta=\{  \exists x\in [\wGamma_q(s)-\delta,\wGamma_q(s)+\delta] \textrm{ admitting almost disjoint geodesics }  \wGamma_{(x,s)}, \wGamma_{(x,s)}^\theta\}.
\end{equation}
 We now claim that the following holds (see Figure \ref{fig:angles}),
  \begin{align}
    \label{eq:63}
    &\cA_\delta\cap \cC_\delta \cap \{\exists x\in [\wGamma_q(s)-\delta,\wGamma_q(s)+\delta] :  \overline{\gamma}_{(-K,0)}^{(x,s)}, \overline{\gamma}_{(K,0)}^{(x,s)} \textrm{ are almost disjoint}\}\nonumber\\
    &\subseteq \cA_\delta\cap \cC_\delta\cap (\cE_{\delta}^{\theta_\delta}\cup \cE_{\delta}^{-\theta_\delta}).
  \end{align}
  and the goal now is to give a careful justification of the above. First, note that, if on the event $\cA_\delta\cap \cC_\delta$, it happens to be the case that the geodesics $\overline{\gamma}_{(-K,0)}^{(x,s)}$ and $\overline{\gamma}_{(K,0)}^{(x,s)}$ are almost disjoint for some $x\in [\wGamma_q(s)-\delta,\wGamma_q(s)+\delta]$, then by geodesic ordering, the right-most geodesics $\overline{\wGamma}_{(x,s)}^{-\theta_\delta}\lvert_{[0,s]}$ and $\overline{\wGamma}_{(x,s)}^{\theta_\delta}\lvert_{[0,s]}$ must be almost disjoint as well. Not only this, in fact, the entire geodesics $\overline{\wGamma}_{(x.s)}^{-\theta_\delta}$ and $\overline{\wGamma}_{(x.s)}^{\theta_\delta}$ must be almost disjoint. To see this, note that if there were $r<0$ for which $\overline{\wGamma}_{(x.s)}^{-\theta_\delta}(r)=\overline{\wGamma}_{(x.s)}^{\theta_\delta}(r)$, then we could consider the concatenation $\wGamma^\circ_{(x,s)}$ of $\overline{\wGamma}_{(x,s)}^{-\theta_\delta}\lvert_{(-\infty,r]}$ and $\overline{\wGamma}_{(x,s)}^{\theta_\delta}\lvert_{[r,s]}$, and this would be a $(-\theta_\delta)$-directed downward semi-infinite geodesic emanating from $(x,s)$. However, we have $\wGamma^\circ_{(x,s)}(0)=\overline{\wGamma}_{(x,s)}^{\theta_\delta}(0)\geq K$ at time $0$ by \eqref{eq:62}, but this contradicts the definition of $\cC_\delta$ which stipulates that $\wGamma_{(x,s)}^\circ(0)\leq  -K$ since $\wGamma_{(x,s)}^\circ$ is a $(-\theta_\delta)$-directed geodesic and since $|x|\leq m_\delta+\delta$. Thus, we have established the inclusion
  \begin{align}
    \label{eq:170}
     &\cA_\delta\cap \cC_\delta \cap \{\exists x\in [\wGamma_q(s)-\delta,\wGamma_q(s)+\delta] :  \overline{\gamma}_{(-K,0)}^{(x,s)}, \overline{\gamma}_{(K,0)}^{(x,s)} \textrm{ are almost disjoint}\}\nonumber\\
    &\subseteq \cA_\delta\cap \cC_\delta\cap \{\exists x\in [\wGamma_q(s)-\delta,\wGamma_q(s)+\delta] :  \overline{\wGamma}_{(x,s)}^{-\theta_\delta}, \overline{\wGamma}_{(x,s)}^{\theta_\delta} \textrm{ are almost disjoint}\}.
  \end{align}
  \begin{figure}
  \centering
  \includegraphics[width=0.8\linewidth]{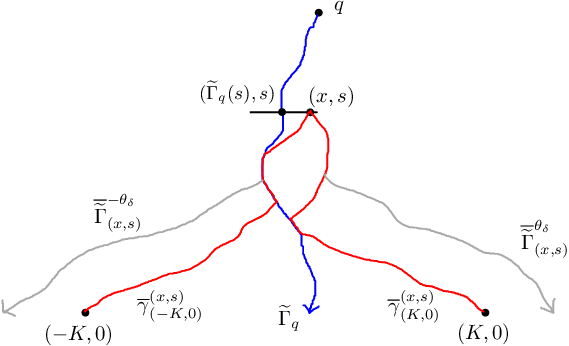}
  \caption{The geodesics involved in the proof of Lemma \ref{lem:14}: If the point $(x,s)$ admits almost disjoint geodesics $\overline{\gamma}_{(-K,0)}^{(x,s)},\overline{\gamma}_{(K,0)}^{(x,s)}$, then on the event $\cA_\delta\cap \cC_\delta$, the geodesics $\overline{\wGamma}_{(x,s)}^{\theta_\delta}$ and $\overline{\wGamma}_{(x,s)}^{-\theta_\delta}$ must also be almost disjoint as noted in \eqref{eq:170}.}
  \label{fig:angles}
\end{figure}
  Now, on the event that $\overline{\wGamma}_{(x,s)}^{-\theta_\delta}$ and $\overline{\wGamma}_{(x,s)}^{\theta_\delta}$ are almost disjoint, we can use Lemma \ref{lem:18} to obtain the existence of an angle $\psi\in [-\theta_\delta,\theta_\delta]$ and a downward $\psi$-directed $\wcL$-geodesic $\wGamma_{(x,s)}^\psi$ satisfying
  \begin{equation}
    \label{eq:169}
    \overline{\wGamma}_{(x,s)}^{-\theta_\delta}(r)\leq \wGamma_{(x,s)}^\psi(r)<\overline{\wGamma}_{(x,s)}^\psi(r)\leq\overline{\wGamma}_{(x,s)}^{\theta_\delta}(r)
  \end{equation}
  for all $r<s$.
Now, if $0\geq \psi$, then we obtain by geodesic ordering that the geodesics $\overline{\wGamma}_{(x,s)}, \overline{\wGamma}_{(x,s)}^{-\theta_\delta}$ are almost disjoint, thereby implying the occurrence of the event $\cE_\delta^{-\theta_\delta}$. Instead, if we had $0<\psi$, then we obtain by geodesic ordering that the geodesics $\underline{\wGamma}_{(x,s)},\overline{\wGamma}_{(x,s)}^{\theta_\delta}$ are almost disjoint, and this implies the occurrence of the event $\cE_\delta^{\theta_\delta}$. As a result of the above and \eqref{eq:170}, we have now established the inclusion in \eqref{eq:63}.

  Now, in view of \eqref{eq:63}, to complete the proof, it suffices to show that for all $\delta$ small enough, we have
  \begin{equation}
    \label{eq:60}
    \PP((\cA_\delta\cap \cC_\delta)^c)+\PP(\cA_\delta\cap \cC_\delta\cap (\cE_{\delta}^{\theta_\delta}\cup \cE_{\delta}^{-\theta_\delta}))\leq \delta^{1-\nu}.
  \end{equation}
  As we noted just after \eqref{eq:62}, we already know that uniformly in the choice of $s\in [\ualpha,\oalpha]$, the probability $1-\PP(\cA_\delta\cap \cC_\delta)$ decays superpolynomially as $\delta\rightarrow 0$, and thus it suffices to show that for all $\delta$ small enough and uniformly over the above choice of $s$, we have
  \begin{equation}
    \label{eq:61}
    \PP(\cE_{\delta}^{\theta_\delta}), \PP(\cE_{\delta}^{-\theta_\delta})\leq \delta^{1-\nu},
  \end{equation}
  but this is precisely the content of Proposition \ref{lem:13}.
  
\end{proof}
The above result yields a good estimate on the probability of the event in the right hand side in Lemma \ref{lem:17}. The goal now is to obtain a good estimate on the probability of the event $\{|\wGamma_q(s)-\Upsilon_{\mathbf{0}}^{X,K}(s)|\leq \delta\}$ conditioned on the above event, and to do so, we shall first need to obtain a $1/2-$ H\"older regularity estimate on the difference profile $\cD_s^{\lambda,K}$ defined for any $s>0$ by
\begin{equation}
  \label{eq:176}
  \cD_s^{\lambda,K}(x)=\cL(\fh_\lambda^K\lvert_{[0,\infty)},0;x,s)-\cL(\fh_\lambda^K\lvert_{(-\infty,0]},0;x,s).
\end{equation}
As is usual for such spatial difference profiles (see for e.g.\ \cite[Proposition 4.1]{RV21}), $\cD_s^{\lambda,K}$ is an increasing function, in the sense that $\cD_s^{\lambda,K}(y)\geq \cD_s^{\lambda,K}(x)$ a.s.\ for all $y\geq x$. We now state the basic H\"older regularity estimate that we shall require.
\begin{lemma}
  \label{lem:10}
 Fix $\ualpha,\oalpha,K,M,\nu>0$ and $q=(y_0,t_0)$ with $t_0>\oalpha$. Then uniformly over the choice of $s\in [\ualpha,\oalpha]$, the probability
  \begin{equation}
    \label{eq:55}
   \PP\left( \sup_{|\lambda|\leq M}\{\cD^{\lambda,K}_s(\wGamma_q(s)+\delta)-\cD^{\lambda,K}_s(\wGamma_q(s)-\delta)\}\geq \delta^{1/2-\nu}\right)
 \end{equation}
 decays superpolynomially as $\delta\rightarrow 0$.
\end{lemma}

\begin{proof}

  By Proposition \ref{lem:12}, we know that there exists a constant $C>0$ such that with superpolynomially high probability, we have $|\wGamma_q(s)|\leq \log^{1/2}\delta^{-1}$.  By using this along with the definition of $\cD^{\lambda,K}_s$, it is easy to see that it suffices to establish that both
  \begin{align}
    \label{eq:81}
    &\PP(\sup_{|x|\leq \log^{1/2}\delta^{-1},|\lambda|\leq M}|\cL(\fh_\lambda^K\lvert_{(-\infty,0]},0;x,s)-\cL(\fh_\lambda^K\lvert_{(-\infty,0]},0;x+2\delta,s)|\geq \delta^{1/2-\nu}),\nonumber\\
    &\PP(\sup_{|x|\leq \log^{1/2}\delta^{-1},|\lambda|\leq M}|\cL(\fh_\lambda^K\lvert_{[0,\infty)},0;x,s)-\cL(\fh_\lambda^K\lvert_{[0,\infty)},0;x+2\delta,s)|\geq \delta^{1/2-\nu})    
  \end{align}
  decay to zero superpolynomially as $\delta\rightarrow 0$. By symmetry, we just establish the latter. To see this, we divide the interval $[K^{-1},K]$ into intervals $I_\delta$ of size $2\delta$ and also divide the set $[-\log^{1/2}\delta^{-1}-2\delta,\log^{1/2}\delta^{-1}+2\delta]$ into intervals $J_\delta$ of size $2\delta$ each.  By Proposition \ref{lem:20} along with a skew invariance (Proposition \ref{prop:14}) argument, it can be seen that uniformly for all intervals $I_\delta,J_\delta$, the probability
  \begin{equation}
    \label{eq:68}
    \PP(\sup_{x\in I_\delta, y\in J_\delta,y'\in J_\delta}|\cL(x,0;y,s)-\cL(x,0;y',s)|\geq \delta^{1/2-\nu})
  \end{equation}
  decays superpolynomially in $\delta$, and then we take a union bound over the $O(\delta^{-1})$ intervals $I_\delta$ and the $O(\delta^{-1} \log^{1/2}\delta^{-1})$ many intervals $J_\delta$. As a result, we have shown that
  \begin{equation}
    \label{eq:107}
   \PP(\bigcup_{\textrm{intervals } I_\delta,J_\delta} \{\sup_{x\in I_\delta, y\in J_\delta,y'\in J_\delta}|\cL(x,0;y,s)-\cL(x,0;y',s)|\geq \delta^{1/2-\nu}\})
 \end{equation}
 decays superpolynomially as $\delta\rightarrow 0$ and it is easy to see that this immediately implies that the probability in \eqref{eq:81} does so as well.
\end{proof}

We are now ready to complete the proof of Proposition \ref{lem:1}.
\begin{proof}[Proof of Proposition \ref{lem:1}]
  For notational convenience, we define the events $\cE_{s,\delta},\cA_{s,\delta}$ by
  \begin{align}
    \label{eq:141}
    \cE_{s,\delta}&= \{|\wGamma_q(s)-\Upsilon_{\mathbf{0}}^{X,K}(s)|\leq \delta\}= \{\cD_s^{X,K}(x)=0 \textrm{ for some } x\in [\wGamma_q(s)-\delta,\wGamma_q(s)+\delta]\},\nonumber\\
    \cA_{s,\delta}&= \{\exists x\in [\wGamma_q(s)-\delta,\wGamma_q(s)+\delta] :  \overline{\gamma}_{(-K,0)}^{(x,s)}, \overline{\gamma}_{(K,0)}^{(x,s)}) \textrm{ are almost disjoint}\},
  \end{align}
   and by Lemma \ref{lem:17}, we know that $\cE_{s,\delta}\subseteq \cA_{s,\delta}$. Define the event $\cH_{s,\delta}$ by
  \begin{equation}
    \label{eq:14}
    \cH_{s,\delta}=\{\sup_{|\lambda|\leq M}\{\cD^{\lambda,K}_s(\wGamma_q(s)+\delta)-\cD^{\lambda,K}_s(\wGamma_q(s)-\delta)\}\geq \delta^{1/2-\nu}\},
  \end{equation}
  and note that by Lemma \ref{lem:10}, $\PP(\cH_{s,\delta})$ decays superpolynomially as $\delta\rightarrow 0$. Now, we write
  \begin{equation}
    \label{eq:13}
    \PP(\cE_{s,\delta})\leq \PP(\cH_{s,\delta})+\PP(\cA_{s,\delta}\cap \cH_{s,\delta}^c)\PP(\cE_{s,\delta} \lvert \cA_{s,\delta}\cap \cH_{s,\delta}^c).
  \end{equation}
  By Lemma \ref{lem:14}, we know that $\PP(\cA_{s,\delta})\leq \delta^{1-\nu}$ for all small enough $\delta$. By using this, we obtain that
  \begin{equation}
    \label{eq:15}
    \PP(\cE_{s,\delta})\leq \PP(\cH_{s,\delta})+ \delta^{1-\nu} \PP(\cE_{s,\delta} \lvert \cA_{s,\delta}\cap \cH_{s,\delta}^c).
  \end{equation}
  We now control $\PP(\cE_{s,\delta}\lvert \cL,\wcL)$. We first note that for any $\lambda_1<\lambda_2\in [-M,M]$ and for all points $(y,t)$ with $t>0$, we have
  \begin{align}
    \label{eq:69}
    \cL( \fh_{\lambda_2}^K\lvert_{[0,\infty)},0; y,t)&\geq \cL( \fh_{\lambda_1}^K\lvert_{[0,\infty)},0; y,t) + 2K^{-1} (\lambda_2-\lambda_1)\nonumber\\
    \cL( \fh_{\lambda_2}^K\lvert_{(-\infty,0]},0; y,t)&\leq \cL( \fh_{\lambda_1}^K\lvert_{(-\infty,0]},0; y,t) - 2K^{-1} (\lambda_2-\lambda_1),                                                             
  \end{align}
  and by subtracting these, we immediately obtain that a.s.\ for all $\lambda_2>\lambda_1$ and all points $(y,t)$ with $t>0$, we have
  \begin{equation}
    \label{eq:111}
    \cD_t^{\lambda_2,K}(y)\geq \cD_t^{\lambda_1,K}(y)+4K^{-1} (\lambda_2-\lambda_1).
  \end{equation}
  As a result of the above, we obtain that on the event $\cH_{s,\delta}^c$, 
  \begin{equation}
    \label{eq:70}
    \mathrm{Leb}\left(\{\lambda\in [-M,M]: \cD_s^{\lambda,K}(x)-\cD_s^{\lambda,K}(x)=0 \textrm{ for some } x\in [\wGamma_q(s)-\delta,\wGamma_q(s)+\delta]\}\right)\leq \delta^{1/2-\nu}/(4K^{-1}),
  \end{equation}
  where we use $\mathrm{Leb}$ to denote the Lebesgue measure.
  Thus, on the event $\cH_{s,\delta}^c$, we have
  \begin{align}
    \label{eq:71}
    \PP(\cE_{s,\delta}\lvert \cL,\wcL)&\leq \PP(X\in \{\lambda: \cD_s^{\lambda,K}(x)=0 \textrm{ for some } x\in [\wGamma_q(s)-\delta,\wGamma_q(s)+\delta]\big \lvert \cL,\wcL)\nonumber\\
                                      &= (2M)^{-1}\mathrm{Leb}\left( \lambda\in [-M,M]: \cD_s^{\lambda,K}(x)=0 \textrm{ for some } x\in [\wGamma_q(s)-\delta,\wGamma_q(s)+\delta]\right)\nonumber\\
    &\leq (2M)^{-1}(K/4)\delta^{1/2-\nu},
  \end{align}
  where we have used that $X$ is independent of $\sigma(\cL,\wcL)$ to obtain the second line.
  Now, on using \eqref{eq:71}, we obtain that
  \begin{equation}
    \label{eq:17}
    \PP(\cE_{s,\delta} \lvert \cA_{s,\delta}\cap \cH_{s,\delta}^c)\leq  KM^{-1}\delta^{1/2-\nu}/8,
  \end{equation}
 and by combining this with \eqref{eq:15}, we have for all small enough $\delta$,
  \begin{equation}
    \label{eq:18}
    \PP(\cE_{s,\delta})\leq \PP(\cH_{s,\delta})+ KM^{-1}\delta^{1-\nu} \delta^{1/2-\nu}/8\leq \delta^{3/2-2\nu},
  \end{equation}
  where in the above, we use the superpolynomial decay of $\PP(\cH_{s,\delta})$ from Lemma \ref{lem:10}. Since $\nu$ is arbitrary, we can replace it by $\nu/2$, and this completes the proof.
\end{proof}

\subsection{Using Proposition \ref{lem:1} to obtain Proposition \ref{lem:9}}
\label{sec:using}
The task now is to use Proposition \ref{lem:1} to prove Proposition \ref{lem:9}. For brevity, for $t'<t$ and paths $\psi,\xi$ defined on an interval containing $[t',t]$, we introduce the notation
\begin{equation}
  \label{eq:84}
  N^{n,\beta}_{t',t}(\xi,\psi)=\#\{i\in [\![t'\varepsilon_n^{-1},t\varepsilon_n^{-1}]\!]: |\xi(i\varepsilon_n)-\psi(i\varepsilon_n)|\leq \varepsilon_n^{2/3-\beta}\}.
\end{equation}
For the remainder of this section, we shall gradually prove stronger and stronger results, and at the end, this shall yield Proposition \ref{lem:9}. To begin, we note the following immediate consequence of Proposition \ref{lem:1}.
\begin{lemma}
  \label{lem:5}
Fix $\oalpha,\ualpha,K,M,\beta>0$ and a point $q=(y_0,t_0)$ with $t_0> \oalpha$. Almost surely, we have 
  \begin{equation}
    \label{eq:85}
    \limsup_{n\rightarrow \infty} \frac{\log N^{n,\beta}_{\ualpha,\oalpha}(\Upsilon_{\mathbf{0}}^{X,K},\wGamma_q)}{\log \varepsilon_n^{-1}}\leq 3\beta/2.
  \end{equation}
 
\end{lemma}
\begin{proof}
  We apply Proposition \ref{lem:1} with $\delta=\varepsilon^{2/3-\beta}_n$ and with $s$ taking values in the discrete set $[\![\ualpha \varepsilon_n^{-1}, \oalpha \varepsilon_n^{-1}]\!]$. Note that the above discrete set has cardinality at most $(\oalpha-\ualpha)\varepsilon_n^{-1}$. By invoking Proposition \ref{lem:1} and taking a union bound over the above set, we obtain that for any fixed $\nu>0$ and for all $\delta$ small enough,
  \begin{align}
    \label{eq:114}
    \EE [N^{n,\beta}_{\ualpha,\oalpha}(\Upsilon_{\mathbf{0}}^{X,K},\wGamma_q)]&\leq (\oalpha-\ualpha)\varepsilon_n^{-1}\times (\varepsilon_n^{2/3-\beta})^{3/2-\nu}\nonumber\\
   &\leq (\oalpha-\ualpha)\varepsilon_n^{-3\beta/2-2\nu/3+\nu\beta}.
  \end{align}
  Thus, by a simple Borel-Cantelli argument, we obtain that almost surely,
  \begin{equation}
    \label{eq:116}
     \limsup_{n\rightarrow \infty} \frac{\log N^{n,\beta}_{\ualpha,\oalpha}(\Upsilon_{\mathbf{0}}^{X,K},\wGamma_q)}{\log \varepsilon_n^{-1}}\leq 3\beta/2+2\nu/3-\nu\beta.
   \end{equation}
   The proof is now completed by noting that $\nu>0$ was chosen arbitrarily.
\end{proof}
The first goal is to obtain an analogous statement where the semi-infinite geodesic $\wGamma_q$ in the above result is replaced by a geodesic $\gamma_p^q$ starting from a fixed point $p=(x_0,0)$ with $x_0\neq 0$. To do so, we shall require the following basic density estimate on $\wGamma_q(0)$.

\begin{lemma}
  \label{lem:7}
  Fix $q=(y_0,t_0)$ with $t_0>0$ and fix $\delta>0,x_0\in \RR$. Almost surely, we have
  \begin{equation}
    \label{eq:44}
    \PP(\wGamma_q(0)\in (x_0-\delta,x_0+\delta)\lvert \cL)>0.
  \end{equation}
\end{lemma}

\begin{proof}
  Recall that $\wGamma_q\lvert_{[0,\infty)}$ can be thought of as the geodesic from $q$ to the initial condition $\wcB_0^0$, where the latter denotes the Busemann function \eqref{eq:102} defined using the landscape $\wcL$. Also, since $\wcL$ was defined by independently resampling the part of $\cL$ below the zero time line, $\wcB_0^0$ is independent of $\cL$. Now, we define the initial conditions $f_\delta=\wcB_0^0\lvert_{(x_0-\delta,x_0+\delta)}$ and $f_\delta^c=\wcB_0^0\lvert_{(x_0-\delta,x_0+\delta)^c}$. Defining $\cL(\fh,0;q)\coloneqq \sup_{x\in \RR}\{\fh(x)+\cL(x,0;q)\}$ for an initial condition $\fh$, it is easy to see that we have
  \begin{equation}
    \label{eq:45}
    \{\wGamma_q(0)\in (x_0-\delta,x_0+\delta)\}=\{\cL(f_\delta,0;q)> \cL(f_\delta^c,0;q)\}. %
  \end{equation}
  In view of the above, it suffices to show that almost surely,
  \begin{equation}
    \label{eq:46}
    \PP(\cL(f_\delta,0;q)> \cL(f_\delta^c,0;q)\lvert \cL,f_\delta^c)>0.
  \end{equation}
  To see this, we note that $\cL(f_\delta,0;q)\geq \max_{|x-x_0|\leq \delta}\widetilde{\cB}_0^0(x)+ \min_{|x-x_0|\leq \delta}\cL(x,0;q)$. Since $\wcB_0^0$ is simply a two-sided Brownian motion (Proposition \ref{prop:22}) independent of $\cL$, we immediately obtain
  \begin{equation}
    \label{eq:47}
    \PP( \max_{|x-x_0|\leq \delta}\wcB_0^0(x)> \cL(f_\delta^c,0;q)-\min_{|x-x_0|\leq \delta}\cL(x,0;q)\lvert \cL, f_\delta^c)>0,
  \end{equation}
  and this completes the proof.
\end{proof}
With the above at hand, we can now obtain the following lemma.
\begin{lemma}
  \label{lem:6}
  Fix $\oalpha,\ualpha,K,M,\beta>0$, a point $p=(x_0,0)$ with $x_0\neq 0$, and a point $q=(y_0,t_0)$ with $t_0>\oalpha$. Almost surely, we have %

\begin{equation}
    \label{eq:41}
    \limsup_{n\rightarrow \infty} \frac{N^{n,\beta}_{\ualpha,\oalpha}(\Upsilon_{\mathbf{0}}^{X,K},\gamma_p^q)}{\log \varepsilon_n^{-1}}\leq 3\beta/2.
  \end{equation}
  
\end{lemma}

\begin{proof}

 Observe that almost surely, there must exist a random $\delta>0$ for which, for any geodesic $\gamma_{(x,0)}^q$ with $|x-x_0|\leq \delta$, we have $\gamma_{(x,0)}^q(s)=\gamma_p^q(s)$ for all $s\in [\ualpha,\oalpha]$. This follows by the a.s.\ uniqueness of the geodesic $\gamma_p^q$ along with Propositions \ref{prop:18}, \ref{prop:19}. As a result of this observation, it suffices to show that for any fixed $\delta>0$, there almost surely exists an $x$ with $|x-x_0|<\delta$ and a geodesic $\gamma_{(x,0)}^q$ for which
  \begin{equation}
    \label{eq:43}
     \limsup_{n\rightarrow \infty} \frac{N^{n,\beta}_{\ualpha,\oalpha}(\Upsilon_{\mathbf{0}}^{X,K},\gamma_{(x,0)}^q)}{\log \varepsilon_n^{-1}}\leq 3\beta/2.
  \end{equation}
 Now, we note that the geodesic $\gamma_p^q$ is measurable with respect to $\cL$ and the interface $\Upsilon_{\mathbf{0}}^{X,K}$ is measurable with respect to $\sigma(\cL,X)$, where we recall that $X$ is independent of $\sigma(\cL,\wcL)$. Thus, as a consequence, \eqref{eq:43} follows immediately on invoking Lemma \ref{lem:5} and Lemma \ref{lem:7}. 
\end{proof}
The goal now is to get rid of the constant $K$ in the above results. For this purpose, we define the initial condition $\fh_X^\infty(x)=\cB_0^0(x)+2Xx$ for all $x\in \RR$, and define $\Upsilon_{\mathbf{0}}^{X,\infty}$ as the a.s.\ unique interface (see Proposition \ref{prop:16}) emanating from the point $\mathbf{0}$ corresponding to the above initial condition. We now have the following result. 
\begin{lemma}
  \label{lem:8}
    Fix $\oalpha,\ualpha,M,\beta>0$, a point $p=(x_0,0)$ with $x_0\neq 0$, and a point $q=(y_0,t_0)$ with $t_0>\oalpha$. Almost surely, we have
  \begin{equation}
    \label{eq:48}
     \limsup_{n\rightarrow \infty} \frac{N^{n,\beta}_{\ualpha,\oalpha}(\Upsilon_{\mathbf{0}}^{X,\infty},\gamma_p^q)}{\log \varepsilon_n^{-1}}\leq 3\beta/2.
  \end{equation}
\end{lemma}
\begin{proof}
  In view of Lemma \ref{lem:6}, it suffices to show that for all $K$ large enough, we have $\Upsilon_{\mathbf{0}}^{X,\infty}(s)=\Upsilon_{\mathbf{0}}^{X,K}(s)$ for all $s\in [\ualpha,\oalpha]$. To see this, consider the left-most and right-most geodesics $\underline{\gamma}_{\ualpha}$ and $\overline{\gamma}_{\ualpha}$ from the point $(\Upsilon_{\mathbf{0}}^{X,\infty}(\ualpha),\ualpha)$ to the initial condition $\fh_X^\infty$. Since $\Upsilon_{\mathbf{0}}^{X,\infty}$ is defined as the interface from $\fh_X^\infty$ emanating from $\0$, we know by Lemma \ref{lem:19} that $\underline{\gamma}_{\ualpha}$ and $\overline{\gamma}_{\ualpha}$ must be almost disjoint and must satisfy $\underline{\gamma}_{\ualpha}(0)<0<\overline{\gamma}_{\ualpha}(0)$. We similarly define $\underline{\gamma}_{\oalpha}$ and $\overline{\gamma}_{\oalpha}$. Now, for any $K$ satisfying $K^{-1}\leq \min(|\underline{\gamma}_{\ualpha}(0)|,|\overline{\gamma}_{\ualpha}(0)|)$ and  $K\geq \max(|\underline{\gamma}_{\oalpha}(0)|,|\overline{\gamma}_{\oalpha}(0)|)$, we note that for all $s\in [\ualpha,\oalpha]$, any right-most geodesic from $(\Upsilon_0^{X,\infty}(s),s)$ to $\fh_X^\infty$ is still a geodesic to the initial condition $\fh_X^K$. Similarly, for all $s\in [\ualpha,\oalpha]$, any left-most geodesic from $(\Upsilon_0^{X,\infty}(s),s)$ to $\fh_X^\infty$ is still a geodesic to the initial condition $\fh_X^K$. Due to this, it is not difficult to see that $\cD_s^{X,K}(\Upsilon_{\mathbf{0}}^{X,\infty}(s))=0$ for all such $s$ and thus $\Upsilon_{\mathbf{0}}^{X,\infty}\lvert_{[\ualpha,\oalpha]}=\Upsilon_{\mathbf{0}}^{X,K}\lvert_{[\ualpha,\oalpha]}$ for all such $K$, and this completes the proof.
\end{proof}
Now, we transition from the interface $\Upsilon_{\0}^{X,\infty}$ to the interface $\Upsilon_{\0}^X$.
\begin{lemma}
  \label{lem:16}
  Fix $\oalpha,\ualpha,M,\beta>0$, a point $p=(x_0,0)$ with $x_0\neq 0$, and a point $q=(y_0,t_0)$ with $t_0>\oalpha$. Almost surely, we have
  \begin{equation}
    \label{eq:72}
     \limsup_{n\rightarrow \infty} \frac{N^{n,\beta}_{\ualpha,\oalpha}(\Upsilon_{\mathbf{0}}^{X},\gamma_p^q)}{\log \varepsilon_n^{-1}}\leq 3\beta/2.
  \end{equation}

\end{lemma}
\begin{proof}
  By Lemma \ref{lem:33}, we note that $\Upsilon_{\mathbf{0}}^{X}$ is the interface starting from $\mathbf{0}$ for the initial condition given by $\cB_X^0$ which has precisely the same law (Proposition \ref{prop:22}) as $\fh_X^\infty$, and $\Upsilon_0^{X,\infty}$ is the interface emanating from $\mathbf{0}$ for $\fh_X^\infty$. The proof is now completed by invoking Lemma \ref{lem:8}.
\end{proof}
Now, we use an approximation argument to go from working with a fixed point $q$ to a more general class of points.
\begin{lemma}
  \label{lem:34}
  Fix $\oalpha,\ualpha,M,\beta>0$ and a point $p=(x_0,0)$ with $x_0\neq 0$. Then almost surely, simultaneously for all points $q=(y,t)\notin \Upsilon_{\0}^X$ with $t>\oalpha$, and all geodesics $\gamma_p^q$, we have
  \begin{equation}
    \label{eq:171}
     \limsup_{n\rightarrow \infty} \frac{N^{n,\beta}_{\ualpha,\oalpha}(\Upsilon_{\mathbf{0}}^{X},\gamma_p^q)}{\log \varepsilon_n^{-1}}\leq 3\beta/2.
  \end{equation}
\end{lemma}
\begin{proof}

  Since $p,q\notin \Upsilon^X_{\mathbf{0}}$ almost surely and since $\gamma_p^q$ and $\Upsilon^X_{\mathbf{0}}$ are continuous, we can find a $\delta>0$ such that $\gamma_p^q(s)\neq \Upsilon^X_{\mathbf{0}}(s)$ for all $s\in [0,\delta]\cup [t-\delta,t]$. We can now choose (see (1) in Proposition \ref{prop:23}) a rational point $\widetilde q$ close to $q$ such that $\gamma_{p}^{\widetilde q}(s)=\gamma_p^q(s)$ for all $s\in [0,t-\delta]$, and as a result, the quantity on the left hand side in \eqref{eq:171} does not change if we replace $q$ by $\widetilde q$. Using the above observation and applying Lemma \ref{lem:16} now completes the proof.
\end{proof}
  By a skew-invariance argument, we can now get rid of the extra randomness coming from $X$.
  \begin{lemma}
    \label{lem:21}
    Fix $\oalpha,\ualpha,\beta>0$ and a point $p=(x_0,0)$ with $x_0\neq 0$. Then almost surely, for all points $q=(y,t)\notin \Upsilon_{\0}$ with $t>\oalpha$, and all geodesics $\gamma_p^q$, we have
    \begin{equation}
      \label{eq:86}
           \limsup_{n\rightarrow \infty} \frac{N^{n,\beta}_{\ualpha,\oalpha}(\Upsilon_{\mathbf{0}},\gamma_p^q)}{\log \varepsilon_n^{-1}}\leq 3\beta/2.
    \end{equation}
  \end{lemma}
  \begin{proof}
    In this proof, we shall use the skew invariance of the directed landscape from Proposition \ref{prop:14}, and we recall the landscape $\cL_\theta^{\mathrm{sk}}$ defined therein for any fixed $\theta\in \RR$. Now, using that $X$ is independent of $\cL$, we also know that $\cL^{\mathrm{sk}}_{X}\stackrel{d}{=}\cL$. Locally, we now use $\Upsilon_{\mathbf{0}}^{0,\mathrm{sk}}$ to denote the $0$-directed interface emanating from $\mathbf{0}$ for $\cL^{\mathrm{sk}}_{X}$. Defining the function $\pi_\theta(x,s)=(x+\theta s,s)$, it can be seen (see \cite[Lemma 7]{Bha23}) that we almost surely have the equality $\pi_X^{-1}(\Upsilon_{\mathbf{0}}^X)=\Upsilon_{\mathbf{0}}^{0,\mathrm{sk}}$ and that for any $\cL$-geodesic $\gamma_p^q$, the path
    \begin{equation}
      \label{eq:173}
      \gamma_{\pi_X^{-1}(p)}^{\pi_X^{-1}(q),\mathrm{sk}}\coloneqq \pi_X^{-1}(\gamma_p^q)
    \end{equation}
    is an $\cL_X^{\mathrm{sk}}$-geodesic between the points $\pi_X^{-1}(p)$ and $\pi_X^{-1}(q)$ and vice-versa. Now, Lemma \ref{lem:21} along with the above yields that almost surely, for all points $p,q$ as in the statement of the lemma, we have
    \begin{equation}
      \label{eq:172}
     \limsup_{n\rightarrow \infty} \frac{N^{n,\beta}_{\ualpha,\oalpha}(\Upsilon_{\mathbf{0}}^{0,\mathrm{sk}},\gamma_{\pi_X^{-1}(p)}^{\pi_X^{-1}(q),\mathrm{sk}})}{\log \varepsilon_n^{-1}}=\limsup_{n\rightarrow \infty} \frac{N^{n,\beta}_{\ualpha,\oalpha}(\Upsilon_{\mathbf{0}}^{X},\gamma_p^q)}{\log \varepsilon_n^{-1}}\leq  3\beta/2.
    \end{equation}
Now, we note that $\pi_X^{-1}(p)=p$ and that $\pi_X^{-1}(q)\notin \Upsilon^{0,\mathrm{sk}}_{\0}$ if and only if $q\notin \Upsilon_{\0}^X$. Finally, since $\cL_\theta^{\mathrm{sk}}\stackrel{d}{=} \cL$ (Proposition \ref{prop:14}), the proof is complete. %
  \end{proof}

  By taking a union over rational $\ualpha<\oalpha$ and invoking the translation invariance of the directed landscape (Proposition \ref{prop:14}), the above immediately implies the following result.
  \begin{lemma}
    \label{lem:22}

    Almost surely, simultaneously for all rational points $w=(x,t')$, all rational points $p=(y',t')\neq w$ and all points $q=(y,t)\notin \Upsilon_{\0}$ with $t>t'$, and all $s',s$ satisfying $t'<s'<s<t$, we have for all geodesics $\gamma_p^q$,
    \begin{equation}
      \label{eq:90}
       \limsup_{n\rightarrow \infty} \frac{N^{n,\beta}_{s',s}(\Upsilon_{w},\gamma_p^q)}{\log \varepsilon_n^{-1}}\leq 3\beta/2.
     \end{equation}
  \end{lemma}
We are finally ready to complete the proof of Proposition \ref{lem:9}.
\begin{proof}[Proof of Proposition \ref{lem:9}]
  \begin{figure}
    \centering
    \includegraphics[width=0.4\linewidth]{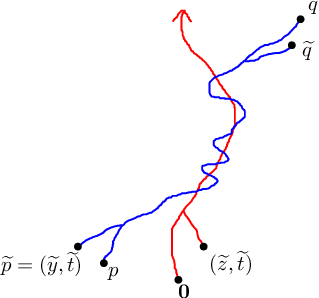}
    \caption{Proof of Proposition \ref{lem:9}; Here, the red curves represent the interfaces $\Upsilon_{\0}$ and $\Upsilon_{(\widetilde{z},\widetilde{t})}$ while the blue curves represent the geodesics $\gamma_p^q$ and $\gamma_{\widetilde{p}}^{\widetilde{q}}$. The rational points $\widetilde{p}$ and $(\widetilde{z},\widetilde{t})$ are chosen on the same time line with the property that the set of times where $\gamma_{\widetilde{p}}^{\widetilde{q}}$ and $\Upsilon_{(\widetilde{z},\widetilde{t})}$ intersect is equal to the set of times that $\gamma_{p}^{q}$ and $\Upsilon_{\0}$ intersect. This allows us to invoke Lemma \ref{lem:22} and complete the proof.}
        \label{fig:sameline1}
  \end{figure}
 A summary of the proof can be found in Figure \ref{fig:sameline1}. Given $p=(y',t'),q=(y,t)\notin \Upsilon_{\0}$ with $0<t'<t$ and a geodesic $\gamma_p^q$, we can first find a rational $\delta>0$ such that $\gamma_p^q(r)\neq \Upsilon_{\0}(r)$ for all $r\in [t',t'+\delta]\cup [t-\delta,t]$. Further, we can use (2) in Proposition \ref{prop:23} to find rational points $\widetilde{p}=(\widetilde{y},\widetilde{t}), \widetilde{q}$ such that
  \begin{equation}
    \label{eq:175}
    \gamma_{\widetilde{p}}^{\widetilde{q}}\lvert_{[t'+\delta,t-\delta]}=\gamma_{p}^q\lvert_{[t'+\delta,t-\delta]}.
  \end{equation}
 Now, by (4) in Proposition \ref{prop:23}, we choose a rational $\widetilde{z}$ close enough to $\Upsilon_{\0}(\widetilde{t})$ such that we have
  \begin{equation}
    \label{eq:174}
    \Upsilon_{(\widetilde{z},\widetilde{t})}\lvert_{[t'+\delta,\infty)}= \Upsilon_{\0}\lvert_{[t'+\delta,\infty)}.
  \end{equation}

  Now, for any choice of the rationals $\delta,\widetilde{y},\widetilde{t}$ satisfying the above, and for all $n$ large enough such that
  \begin{equation}
    \label{eq:118}
    \varepsilon_n^{2/3-\beta}<\inf_{t\in[t',t'+\delta]\cup [t-\delta,t]}|\gamma_p^q(t)- \Upsilon_{\0}(t)|,
  \end{equation}
  we have
  \begin{equation}
    \label{eq:117}
   N^{n,\beta}_{t',t}(\Upsilon_{\0},\gamma_{ p}^{q})= N^{n,\beta}_{t'+\delta,t-\delta}(\Upsilon_{(\widetilde z, \widetilde t)},\gamma_{\widetilde{p}}^{\widetilde{q}}),
  \end{equation}
 and using this along with Lemma \ref{lem:22} completes the proof.

\end{proof}
\section{Constructing good paths which intersect $\Upsilon_{\0}$ finitely often}
\label{sec:good-paths}
The goal of this section is to prove the following result.
\begin{proposition}
  \label{prop:6}
  Fix points $p=(x_0,s_0)$ and $q=(y_0,t_0)$ with $0<s_0<t_0$. Then almost surely, there exists a sequence of paths $\gamma_n$ from $p$ to $q$ such that we have the following:
  \begin{enumerate}
  \item For all $n$, the times $s$ for which $\gamma_n(s)=\Upsilon_{\0}(s)$ form a finite set.
  \item We have the convergence $\ell(\gamma_n;\cL)\rightarrow \ell(\gamma_p^q;\cL)$ as $n\rightarrow \infty$.
  \end{enumerate}
\end{proposition}

The proof of the above is technically complicated, and involves a local modification argument as outlined in Section \ref{sec:outline}. Before beginning the construction of the paths $\gamma_n$ appearing in Proposition \ref{prop:6}, we first present a technical uniform estimate on the transversal fluctuation and lengths of semi-infinite geodesics (see condition (1) in Section \ref{sec:outline}), and this will be important in the proof of Proposition \ref{prop:6}. We note that in this estimate and also later in this section, we shall often use $C$s, possibly with subscripts, to denote \textit{random constants}, by which we mean positive random variables measurable with respect to $\cL$-- the crucial point being that these variables will have no dependence on $n$.

\begin{lemma}
  \label{lem:25}
  Fix $M>0,\nu>0$ and consider the set $K=[-M,M]^2\subseteq \RR^2$. Almost surely, for a random constant $C$, we have for all $(y,t)\in K$ and all geodesics $\Gamma_{(y,t)}$,
  \begin{align}
    \label{eq:128}
    \sup_{[s,s']\subseteq  [-2M,t]}|\Gamma_{(y,t)}(s)-\Gamma_{(y,t)}(s')|&\leq C|s'-s|^{2/3-\nu}\\
    \label{eq:1281}
    \sup_{[s,s']\subseteq [-2M,t]}|\ell(\Gamma_{(y,t)}\lvert_{[s,s']};\cL)|&\leq C|s'-s|^{1/3-\nu}.
    \end{align}
\end{lemma}

\begin{proof}
   By a simple transversal fluctuation argument (e.g.\ by using Proposition \ref{lem:12}), it can be seen that for large enough $m\in \NN$, we have
  \begin{equation}
    \label{eq:127}
    \Gamma_{(-m,M)}(r)< \Gamma_{(y,t)}(r)< \Gamma_{(m,M)}(r)
  \end{equation}
  for all $r\in [-2M,t]$, points $(y,t)\in K$ and geodesics $\Gamma_{(y,t)}$.
  As a result of this observation, to obtain \eqref{eq:128}, it suffices to establish that for any fixed $L>0$, there exists a random constant $C$ such that for all $|x|\leq L$, all points $(y,t)\in K$ and all geodesics $\gamma_{(x,-2M)}^{(y,t)}$, we almost surely have
  \begin{equation}
    \label{eq:123}
     \sup_{[s,s']\subseteq [-2M,t]}|\gamma_{(x,-2M)}^{(y,t)}(s)-\gamma_{(x,-2M)}^{(y,t)}(s')|\leq C|s'-s|^{2/3-\nu},
  \end{equation}
   but this follows immediately by Proposition \ref{prop:17}. The task now is to obtain \eqref{eq:1281}. In view of the above, we note that it suffices to show that for some random constant $C_1$, for all $|x|\leq L$, all points $(y,t)\in K$ and all geodesics $\gamma_{(x,-2M)}^{(y,t)}$, we almost surely have
  \begin{equation}
    \label{eq:129}
    \sup_{[s,s']\subseteq [-2M,t]}|\ell(\gamma_{(x,-2M)}^{(y,t)}\lvert_{[s,s']};\cL)|\leq C_1|s'-s|^{1/3-\nu}.
  \end{equation}
 Now, by noting that $\ell(\gamma_{(x,-2M)}^{(y,t)}\lvert_{[s,s']};\cL)$ is simply equal to $\cL(\gamma_{(x,-2M)}^{(y,t)}(s),s;\gamma_{(x,-2M)}^{(y,t)}(s'),s')$, we can use Proposition \ref{prop:7} and \eqref{eq:123} to obtain that almost surely for random constants $C_2,C_3,C_4$, 
\begin{align}
  \label{eq:126}
  |\ell(\gamma_{(x,-2M)}^{(y,t)}\lvert_{[t-\delta,t]};\cL)|&\leq (C_2|s'-s|^{2/3-\nu})^2/|s'-s|+C_3|s'-s|^{1/3-\nu}\nonumber\\
  &\leq C_4|s'-s|^{1/3-2\nu}.
 \end{align}
The proof is now completed by replacing $\nu$ by $\nu/2$.
\end{proof}

We are now ready to start constructing the paths $\gamma_n$ and we shall always work with a fixed $(p;q)=(x_0,s_0;y_0,t_0)\in \RR_\uparrow^4$ as in the statement of Proposition \ref{prop:6}. For $i\in [\![\lfloor s_0\varepsilon_n^{-1}\rfloor ,\lfloor t_0 \varepsilon_n^{-1}\rfloor ]\!]$, define the interval $I_i=[i\varepsilon_n,(i+1)\varepsilon_n]$, and note that we have suppressed the dependency on $n$ in the notation here. We say that an interval $I_i$ as above is bad if the event
    \begin{equation}
      \label{eq:23}
      \{\gamma_p^q(s)=\Upsilon_{\0}(s) \textrm{ for some } s\in I_i\}
    \end{equation}
    occurs, and we shall use $\cI_n$ to denote the set of $i$ such that $I_i$ is bad. If an interval $I_i$ is not bad, then we simply call it good. Now, since $p,q$ are fixed points, we have $p,q\notin \Upsilon_{\0}$ almost surely. As a result, it is easy to see that for all $n$ large enough, the intervals $I_{\lfloor s_0\varepsilon_{n}^{-1}\rfloor}, I_{\lfloor t_0\varepsilon_{n}^{-1}\rfloor}$ are both good, and throughout this section, we shall always implicitly assume that $n$ is taken to be large enough so that the above holds. Now, as an immediate consequence of Proposition \ref{prop:10}, we have the following result.
    \begin{lemma}
      \label{lem:26}
      Almost surely, we have $\frac{\log \#\cI_n}{\log \varepsilon_n^{-1}}\rightarrow 0$ as $n\rightarrow \infty$.
    \end{lemma}

    We now split the set of bad intervals $I_i$ into collections of consecutive intervals. Indeed, we define a set $\cJ_n\subseteq \cI_n$ such that $j\in \cJ_n$ if and only $I_j$ is bad and $I_{j-1}$ is good.
 Now, we define $j^*=\min_{i\geq j}\{I_{i+1} \textrm{ is good}\}$.
 Finally, we define $J_j$ for $j\in \cJ_n$ by
    \begin{equation}
      \label{eq:24}
      J_j=\bigcup_{i\in [\![j,j^*]\!]}I_i.
    \end{equation}
It is easy to see that we have the following result.
    \begin{lemma}
      \label{lem:27}
   For different values of $j\in \cJ_n$, the intervals $J_j$ are disjoint. Also, we have the equality $\bigcup_{j\in \cJ_n}J_j=\bigcup_{i\in \cI_n} I_i$
    \end{lemma}
    With the intervals $J_j$ at hand, we are now ready to define the paths $\gamma_n$, and we refer the reader to Figure \ref{fig:mod} for a depiction of this construction. First, we require that $\gamma_n$ and $\gamma_p^q$ agree outside $\bigcup_{j\in \cJ_n}(I_{j-1}\cup J_j)$, that is, for $s\in (\bigcup_{j\in \cJ_n}(I_{j-1}\cup J_j))^c\cap [s_0,t_0]$, we simply define \begin{equation}
  \label{eq:25}
  \gamma_n(s)=\gamma_p^q(s).
\end{equation}
\begin{figure}
  \centering
  \includegraphics[width=0.5\linewidth]{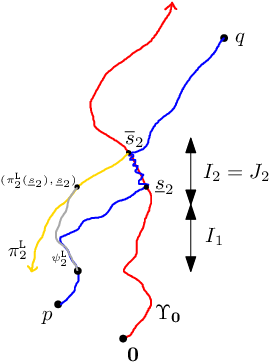}
  \caption{An overview of the definition of $\gamma_n$: Consider the shown setting where $n=2$ and thus $\varepsilon_n=1/4$. The blue path is the a.s.\ unique geodesic $\gamma_p^q$ and the red path denotes the interface $\Upsilon_{\0}$. Here, there are only five intervals $I_0,I_1,I_2,I_3,I_4$ out of which only the interval $I_2$ is bad. Everywhere except on $I_1\cup J_2$, we set $\gamma_n$ to be equal to $\gamma_{p}^{q}$. To define $\gamma_n$ on $I_1\cup J_2$, we first define the golden path to be the left-most downward $0$-directed geodesic emanating from $(\Upsilon_{\0}(\overline{s}_2),\overline{s}_2)$ and note that this stays to the left of $\Upsilon_{\0}$. Further, we consider a geodesic (grey) $\psi_2^{*_2}=\psi_2^{\tL}$ from $(\gamma_p^q(\varepsilon_n),\varepsilon_n)$ to $(\pi_2^{\tL}(\underline{s}_2),\underline{s}_2)$ and this geodesic stays to the left of the blue path by geodesic ordering. Now, we define $\gamma_n$ on $I_1\cup J_2$ by first following the grey geodesic till time $\underline{s}_2$ followed by using the golden path till time $\overline{s}_2$ and then subsequently tracing the blue path.}
\label{fig:mod}
\end{figure}
As a result, by the definition of good intervals, for all $s\in (\bigcup_{j\in \cJ_n}(I_{j-1}\cup J_j))^c\cap [s_0,t_0]$, we have $\gamma_n(s)\neq \Upsilon_{\0}(s)$.

The task now is to fix $j\in \cJ_n$ and define $\gamma_n(s)$ for $s\in I_{j-1}\cup J_j$. To do so, we define the times $\overline{s}_j,\underline{s}_j$ by
\begin{align}
  \label{eq:26}
  \overline{s}_j= \sup_{t\in J_j}\{\gamma_p^q(t)=\Upsilon_{\0}(t)\} ,\nonumber\\
  \underline{s}_j= \inf_{t\in J_j}\{\gamma_p^q(t)=\Upsilon_{\0}(t)\},
\end{align}
and we note that $\gamma_p^q\lvert_{[ (j-1)\varepsilon_n,\underline{s}_j)}\subseteq \cS_{\Upsilon_{\0}}^{*_j}$ for exactly one $*_j\in \{\tL,\tR\}$ since the interval $I_{j-1}$ is good.
Now, we consider the two paths $\pi_j^{\tL}$ and $\pi_j^{\tR}$ defined as
\begin{equation}
  \label{eq:119}
  \pi_j^{\tL}=\underline{\Gamma}_{(\Upsilon_{\0}(\overline{s}_j),\overline{s}_j)}\lvert_{[0,\overline{s}_j]}, \pi_j^{\tR}=\overline{\Gamma}_{(\Upsilon_{\0}(\overline{s}_j),\overline{s}_j)}\lvert_{[0,\overline{s}_j]}
\end{equation}
As a consequence of Lemmas \ref{lem:33}, \ref{lem:19}, we crucially note that $\pi_j^\tL\lvert_{[0,\overline{s}_j)}\subseteq \cS_{\Upsilon_{\0}}^{\tL}$ and $\pi_j^\tR\lvert_{[0,\overline{s}_j)}\subseteq \cS_{\Upsilon_{0}}^{\tR}$. We are now ready to define $\gamma_n(s)$ for $s\in I_{j-1}\cup J_j$ with $j\in \cJ_n$.%

\subsubsection*{\textbf{Defining $\gamma_n$ on $I_{j-1}\cup J_j$}}
Let $\psi_j^{*_j}$ be a geodesic from $(\gamma_p^q((j-1)\varepsilon_n),(j-1)\varepsilon_n)$ to $(\pi_j^{*_j}(\underline{s}_j),\underline{s}_j)$. Since $\gamma_p^q\lvert_{[ (j-1)\varepsilon_n,\underline{s}_j)}\subseteq \cS_{\Upsilon_{\0}}^{*_j}$ and since the geodesic $\gamma_p^q$ is a.s.\ unique, by geodesic ordering, we have $\psi_j^{*_j}\subseteq \cS_{\Upsilon_{\0}}^{*_j}$ as well. With the above in mind, we now define $\gamma_n(t)$ for $t\in I_{j-1}\cup J_j$ by
\begin{equation}
  \label{eq:28}
  \gamma_n(t)=
  \begin{cases}
    \gamma_p^q(t) &t>\overline{s}_j,\\
    \pi_j^{*_j}(t) &t\in  [\underline{s}_j,\overline{s}_j],\\
    \psi_j^{*_j}(t) &t<\underline{s}_j.
  \end{cases}
\end{equation}
This completes the definition of the paths $\gamma_n$. %
The following lemma is now immediate from the way in which $\gamma_n$ has been defined.
\begin{lemma}
  \label{lem:28}
  We have the equality
  \begin{equation}
    \label{eq:142}
    \{s\in [s_0,t_0]: \gamma_n(s)=\Upsilon_{\0}(s)\}=\bigcup_{j\in \cJ_n} \{\overline{s}_j\},
  \end{equation}
and thus the set on the left hand side above is a.s.\ finite for all $n$.
\end{lemma}

In view of the above lemma and the statement of Proposition \ref{prop:6}, our aim now is to establish the following two results.
\begin{lemma}
  \label{lem:29}
  Almost surely, as $n\rightarrow \infty$, we have
  \begin{equation}
    \label{eq:143}
    |\sum_{j\in \cJ_n} \ell(\gamma_p^q\lvert_{[ (j-1)\varepsilon_n,\overline{s}_j]};\cL)|\rightarrow 0
  \end{equation}
\end{lemma}

\begin{lemma}
  \label{lem:30}
  Almost surely, as $n\rightarrow \infty$, we have
  \begin{equation}
    \label{eq:144}
    |\sum_{j\in \cJ_n} \ell (\gamma_n\lvert_{[ (j-1)\varepsilon_n,\overline{s}_j]};\cL)|\rightarrow 0.
  \end{equation}
\end{lemma}

Indeed, we first use the above two results to quickly complete the proof of Proposition \ref{prop:6}.
\begin{proof}[Proof of Proposition \ref{prop:6} assuming Lemmas \ref{lem:29}, \ref{lem:30}]
  As a consequence or Lemma \ref{lem:28}, we already know that for all $n$, the set of times $s$ for which $\gamma_n(s)=\Upsilon_{\0}(s)$ is finite. Thus, we need only show that $\ell(\gamma_n;\cL)\rightarrow \ell(\gamma_p^q;\cL)$ as $n\rightarrow \infty$. Now, from the definition of $\gamma_n$, we know that $\gamma_n$ and $\gamma_p^q$ only possibly differ on the set $\bigcup_{j\in \cJ_n}[ (j-1)\varepsilon_n,\overline{s}_j]$, and as a result, we know that
  \begin{align}
    \label{eq:145}
    |\ell(\gamma_p^q;\cL)-  \ell(\gamma_n;\cL)|&=|\sum_{j\in \cJ_n} \ell(\gamma_p^q\lvert_{ [ (j-1)\varepsilon_n,\overline{s}_j]};\cL)- \sum_{j\in \cJ_n} \ell (\gamma_n\lvert_{ [ (j-1)\varepsilon_n,\overline{s}_j]};\cL)|\nonumber\\
    &\leq |\sum_{j\in \cJ_n} \ell(\gamma_p^q\lvert_{ [ (j-1)\varepsilon_n,\overline{s}_j]};\cL)|+ |\sum_{j\in \cJ_n} \ell (\gamma_n\lvert_{[ (j-1)\varepsilon_n,\overline{s}_j]};\cL)|,
  \end{align}
  and thus the required convergence now follows by invoking Lemma \ref{lem:29} and Lemma \ref{lem:30}. This completes the proof.
\end{proof}
Now, we seek to establish Lemma \ref{lem:29} and Lemma \ref{lem:30}, and we begin with the former.
\begin{proof}[Proof of Lemma \ref{lem:29}]
 By the a.s.\ H\"older continuity of the geodesic $\gamma_p^q$ (see Proposition \ref{prop:17}) along with Proposition \ref{prop:7}, it is not difficult to see that the function $f(s)=\cL(\gamma_p^q(s),s;q)$ is almost surely $1/3-$ H\"older regular. %
Thus, there exists a random constant $C>0$ such that for any fixed $\nu>0$ and all $n$ large enough, we almost surely have
\begin{align}
  \label{eq:29}
  \sum_{j\in \cJ_n} \ell(\gamma_p^q\lvert_{ [ (j-1)\varepsilon_n,\overline{s}_j]};\cL)&=\sum_{j\in \cJ_n}|f((j-1)\varepsilon_n)-f(\overline{s}_j )|\nonumber\\
                                                              &\leq C\sum_{j\in \cJ_n}( |J_j|+ \varepsilon_n)^{1/3-\nu}\nonumber\\
  &=C\sum_{j\in \cJ_n}(\sum_{i:I_i\subseteq J_j}|I_i|+\varepsilon_n)^{1/3-\nu}\nonumber\\
                                                              &\leq C\sum_{j\in \cJ_n}\sum_{i: I_i\subseteq J_j}|I_i|^{1/3-\nu}+ C\sum_{j\in \cJ_n}\varepsilon_n^{1/3-\nu}\nonumber\\
  &=C(\#\cI_n+\#\cJ_n)\varepsilon_n^{1/3-\nu}\nonumber\\
                                                              &\leq 2C \varepsilon_n^{1/3-\nu}\#\cI_n\nonumber\\
  &\leq 2C\varepsilon_n^{1/3-2\nu},                                                            \end{align}
where we used Lemma \ref{lem:26} to obtain the last inequality. Since $\varepsilon_n\rightarrow 0$ as $n\rightarrow \infty$, this completes the proof.
\end{proof}
We now provide the proof of Lemma \ref{lem:30}, and this is more involved.
  \begin{proof}[Proof of Lemma \ref{lem:30}]

  We first show that a.s.\ for any fixed $\nu>0$, for all $j\in \cJ_n$ and all $n$ large enough, we have
\begin{equation}
  \label{eq:31}
  |\ell(\gamma_n\lvert_{[ (j-1)\varepsilon_n,\overline{s}_j]};\cL)|\leq C_1|J_j|^{1/3-\nu} + C_2\varepsilon_n^{-1}|J_j|^{4/3-2\nu},
\end{equation}
for some random constants $C_1, C_2$. %

Since $\Upsilon_{\0}\lvert_{[0,t_0]}$ is a compact set, we know by Lemma \ref{lem:25} that for a random constant $C_1$,
\begin{equation}
  \label{eq:36}
  |\ell(\pi_{j}^{*_j}\lvert_{  [\underline{s}_j,\overline{s}_j]};\cL)|\leq C_1(\overline{s}_j-\underline{s}_j)^{1/3-\nu}\leq C_1|J_j|^{1/3-\nu},
\end{equation}
and further, by using \eqref{eq:128} from Lemma \ref{lem:25}, we know that for a random constant $C_2$, 
\begin{equation}
  \label{eq:130}
  |\pi_j^{*_j}(\underline{s}_j)-\Upsilon_{\0}(\overline{s}_j)|=|\pi_j^{*_j}(\underline{s}_j)-\pi_j^{*_j}(\overline{s}_j)|\leq C_2 ( \overline{s}_j-\underline{s}_j)^{2/3-\nu}\leq C_2|J_j|^{2/3-\nu}.
\end{equation}
Similarly, by using the $2/3-$ H\"older continuity of $\gamma_p^q$, for a random constant $C_3$, we have
\begin{equation}
  \label{eq:131}
  |\gamma_p^q((j-1)\varepsilon_n)-\Upsilon_{\0}(\overline{s}_j)|=|\gamma_p^q((j-1)\varepsilon_n)-\gamma_p^q(\overline{s}_j)|\leq C_3 ( \overline{s}_j-(j-1)\varepsilon_n)^{2/3-\nu}\leq C_3(|J_j|+\varepsilon_n)^{2/3-\nu}.
\end{equation}
On combining \eqref{eq:130} and \eqref{eq:131}, we obtain that
\begin{equation}
  \label{eq:132}
  |\pi_j^{*_j}(\underline{s}_j)-\gamma_p^q((j-1)\varepsilon_n)|\leq  (C_2+C_3)(|J_j|+\varepsilon_n)^{2/3-\nu}\leq 2(C_2+C_3)|J_j|^{2/3-\nu}.
\end{equation}

Now, recalling that $\psi_j^{*_j}$ is defined to be a geodesic from $(\gamma_p^q((j-1)\varepsilon_n),(j-1)\varepsilon_n)$ to $(\pi_j^{*_j}(\underline{s}_j),\underline{s}_j)$, we can use \eqref{eq:132} along with Proposition \ref{prop:7} to obtain that for some random constants $C_4,C_5$,
\begin{align}
  \label{eq:37}
  |\ell(\psi_{j}^{*_j}; \cL)|&\leq C_4(\underline{s}_j-(j-1)\varepsilon_n)^{1/3-\nu}+ (\underline{s}_j-(j-1)\varepsilon_n)^{-1}\left(2(C_2+C_3)|J_j|^{2/3-\nu}\right)^2\nonumber\\
  &\leq C_4(|J_j|+\varepsilon_n)^{1/3-\nu}+C_5\varepsilon_n^{-1}|J_j|^{4/3-2\nu},
\end{align}
where we have used the deterministic inequalities $|J_j|+\varepsilon_n\geq \underline{s}_j-(j-1)\varepsilon_n\geq \varepsilon_n$. As a result, by combining the above and \eqref{eq:36}, we obtain that for some random constants $C_6,C_7$,
\begin{align}
  \label{eq:38}
  |\ell(\gamma_n\lvert_{[ (j-1)\varepsilon_n,\overline{s}_j]};\cL)|&\leq |\ell(\pi_{j}^{*_j}\lvert_{ [\underline{s}_j,\overline{s}_j]};\cL)|+ |\ell(\psi_{j}^{*_j}; \cL)|\nonumber\\
  &\leq C_6(|J_j|+\varepsilon_n)^{1/3-\nu} + C_7\varepsilon_n^{-1}|J_j|^{4/3-2\nu}\nonumber\\
  &\leq 2C_6|J_j|^{1/3-\nu} + C_7\varepsilon_n^{-1}|J_j|^{4/3-2\nu},
\end{align}
and this completes the proof of \eqref{eq:31}.

Finally, we now use \eqref{eq:31} to complete the proof of the lemma. Indeed, we can write for any fixed $\nu\in (0,1/6)$,
\begin{align}
  \label{eq:39}
  \sum_{j\in \cJ_n} |\ell (\gamma_n\lvert_{[ (j-1)\varepsilon_n,\overline{s}_j]};\cL)|  &\leq \sum_{j\in \cJ_n} C_1|J_j|^{1/3-\nu}+ C_2\sum_{j\in \cJ_n}\varepsilon_n^{-1}|J_j|^{4/3-2\nu}\nonumber\\
  &\leq C_1\sum_{j\in \cJ_n}\sum_{i:I_I\subset J_j}|I_i|^{1/3-\nu}+C_2\varepsilon_n^{-1}(\sum_{j\in \cJ_n} |J_j|)^{4/3-2\nu}\nonumber\\
                                                                   &= C_1\varepsilon_n^{1/3-\nu}\#\cI_n+ C_2\varepsilon_n^{-1}  (\varepsilon_n\#\cI_n)^{4/3-2\nu}\nonumber\\
  &= C_1\varepsilon_n^{1/3-\nu}\#\cI_n+ C_2\varepsilon_n^{1/3-2\nu}  (\#\cI_n)^{4/3-2\nu}\nonumber\\
                                                                   &\leq C_3\varepsilon_n^{1/3-3\nu}.                                               
\end{align}
To obtain the second line above, we have simply used that $0<1/3-\nu< 1< 4/3-2\nu$. Finally, the last line was obtained by invoking Lemma \ref{lem:26}. Since $\varepsilon_n\rightarrow 0$ as $n\rightarrow \infty$, this finishes the proof.

\end{proof}

\section{Completion of the proofs of the main results}
\label{sec:finish-main}
Since the proof of Theorem \ref{thm:3} was already completed in Section \ref{sec:proof-thm-dim}, it remains to complete the proofs of Theorem \ref{thm:1} and Theorem \ref{thm:2}, and we now begin with the former.
\begin{proof}[Proof of Theorem \ref{thm:1}]
  We seek to establish the a.s.\ equality $\cL_{\Upsilon_{\0}}^g(p;q)=\cL(p;q)$. Now, by definition, we know that
  \begin{equation}
    \label{eq:147}
    \cL_{\Upsilon_{\0}}^g(p;q)\leq \cL(p;q),
  \end{equation}
   and thus we need only establish the reverse inequality. Recall that Proposition \ref{prop:6} shows that almost surely, there exist paths $\gamma_n$ from $p$ to $q$ which intersect $\Upsilon_{\0}$ only finitely many times and further satisfy $\cL(\gamma_n;\cL)\rightarrow \ell(\gamma_p^q;\cL)=\cL(p;q)$ as $n\rightarrow \infty$. Now, we recall the definition of $\cL_{\Upsilon_{\0}}^g$ and note that since $\gamma_n$ intersects $\Upsilon_{\0}$ only finitely many times, it yields a valid partition in the supremum over partitions present in the definition of $\cL_{\Upsilon_{\0}}^g(p;q)$ (see Definition \ref{def:recons}). As a result, we obtain that
  \begin{equation}
    \label{eq:146}
    \cL_{\Upsilon_{\0}}^g(p;q)\geq \ell(\gamma_n;\cL)
  \end{equation}
  for all $n$, and since the right-hand side almost surely converges to $\cL(p;q)$, we obtain that $\cL_{\Upsilon_{\0}}^g(p;q)\geq \cL(p;q)$ almost surely. Thus, on combining with \eqref{eq:147}, we have shown that $\cL_{\Upsilon_{\0}}^g(p;q)=\cL(p;q)$, and this completes the proof.

\end{proof}

We now use the results from \cite{DZ21} summarised in Section \ref{sec:disj-optim} to prove Theorem \ref{thm:2}.
\begin{proof}[Proof of Theorem \ref{thm:2}]
  Suppose that the geodesic $\gamma=\gamma_{(0,0)}^{(0,1)}$ is metrically removable. We show that in this case, for all rational points $p=(x,s),q=(y,t)$ with $0<s<t<1$, we must almost surely have
  \begin{equation}
    \label{eq:150}
    \cL^*((\gamma(s),x),s;(\gamma(t),y),t)=\cL(\gamma(s),s;\gamma(t),t)+\cL(p;q)
  \end{equation}

Indeed, since $\gamma$ has been assumed to be metrically removable, almost surely, for all $p,q$ as in the above, we must have $\cL_\gamma^g(p;q)=\cL(p;q)$. As a result, for any $\delta>0$, there must exist a path $\psi$ from $p$ to $q$ for which
  \begin{equation}
    \label{eq:94}
    \ell(\psi;\cL)\geq \cL(p;q)-\delta,
  \end{equation}
 and further there exist $m\in \NN$ and times $\cR=\{s<r_1\dots<r_{m}<t\}$ such that $\psi(r)=\gamma(r)$ if and only if $r\in \cR$. Thus, we note that for any $i\in [\![1,m-1]\!]$, $\psi\lvert_{[r_i,r_{i+1}]}$ and $\gamma\lvert_{[r_i,r_{i+1}]}$ are almost disjoint paths from $(\gamma(r_{i}),r_i)$ to $(\gamma(r_{i+1}),r_{i+1})$. As a consequence, we obtain that
  \begin{equation}
    \label{eq:92}
    \cL^*((\gamma(r_i),\gamma(r_i)),r_i;(\gamma(r_{i+1}),\gamma(r_{i+1})),r_{i+1})\geq \ell(\gamma\lvert_{[r_i,r_{i+1}]};\cL)+ \ell(\psi\lvert_{[r_i,r_{i+1}]};\cL).
  \end{equation}
 Now, by using the reverse triangle inequality satisfied by the extended directed landscape (Proposition \ref{prop:12}), we obtain that
  \begin{align}
    \label{eq:151}
    \cL^*((\gamma(s),x),s;(\gamma(t),y),t)&\geq \cL^*((\gamma(s),x),s;(\gamma(r_1),\gamma(r_1)),r_1)\nonumber\\
    &+\sum_{i=1}^{m-1}\cL^*((\gamma(r_i),\gamma(r_i)),r_i;(\gamma(r_{i+1}),\gamma(r_{i+1})),r_{i+1})\nonumber\\
    &+ \cL^*((\gamma(r_m),\gamma(r_m)),r_m;(\gamma(t),y),t)
  \end{align}
  On using \eqref{eq:92}, \eqref{eq:151} and further noting that $\gamma\lvert_{[s,r_1]}$ and $\psi\lvert_{[s,r_1]}$ are almost disjoint, and that $\gamma\lvert_{[r_m,t]}$ and $\psi\lvert_{[r_m,t]}$ are almost disjoint as well, we obtain
  \begin{align}
    \label{eq:152}
    \cL^*((\gamma(s),x),s;(\gamma(t),y),t)&\geq  (\ell(\gamma\lvert_{[s,r_1]};\cL)+ \ell(\psi\lvert_{[s,r_1]};\cL))+\sum_{i=1}^{m-1}\left(\ell(\gamma\lvert_{[r_i,r_{i+1}]};\cL)+ \ell(\psi\lvert_{[r_i,r_{i+1}]};\cL)\right)\nonumber\\
    &+ (\ell(\gamma\lvert_{[r_m,t]};\cL)+ \ell(\psi\lvert_{[r_m,t]};\cL)),
  \end{align}
  and this is, of course, the same as
  \begin{equation}
    \label{eq:153}
    \cL^*((\gamma(s),x),s;(\gamma(t),y),t)\geq \ell(\gamma\lvert_{[s,t]};\cL)+\ell(\psi;\cL)\geq \cL(\gamma(s),s;\gamma(t),t)+\cL(p;q)-\delta,
  \end{equation}
  where we have used \eqref{eq:94} to obtain the second inequality. Since $\delta$ was arbitrary, we have now shown \eqref{eq:150}.

  Now, we choose a sequence of rational points $p_n,q_n$ such that $p_n\rightarrow (0,0), q_n\rightarrow (0,1)$ as $n\rightarrow \infty$. Using \eqref{eq:150} for these points along with the continuity of the extended directed landscape (Proposition \ref{prop:24}), we obtain that
  \begin{equation}
    \label{eq:154}
    \cL^*((0,0),0;(0,0),1)=2\cL(0,0;0,1).
  \end{equation}
  However, this in conjunction with Proposition \ref{prop:20} implies that there must exist two almost disjoint geodesics from $(0,0)$ to $(0,1)$, but this is a contradiction since $\gamma=\gamma_{(0,0)}^{(0,1)}$ is a.s.\ unique. As a result, our assumption that $\gamma$ is metrically removable must have been false, and this completes the proof.
\end{proof}

\section{An open question}
\label{sec:open}
Recall the setting of $\gamma$-LQG mentioned briefly in the introduction. The statement of Theorem \ref{thm:3} now suggests the following natural question.
\begin{question}
  Are geodesics in $\gamma$-LQG metrically removable?
\end{question}
Recall that the proof of Theorem \ref{thm:3} was based on from \cite{DZ21}. At least at an intuitive level, undertaking a corresponding study of disjoint optimizers for $\gamma$-LQG seems more delicate than for the directed landscape due to the non-trivial correlations present for the Gaussian free field, in the sense that the environment just off of a $\gamma$-LQG geodesic might not be as ``bad'' as is the case for the environment just off of a directed landscape geodesic (see for e.g.\ \cite{MSZ21,DSV22}, \cite[Section 5.1]{BSS19}, \cite[Lemma 4.17]{GH23}).

\printbibliography
\end{document}